\documentclass[11pt,letterpaper,titlepage]{article}
\usepackage[T1]{fontenc}
\usepackage{lmodern,amsmath,amssymb,amsthm,mathtools,microtype}
\usepackage[margin=1in,headheight=14pt,headsep=7mm]{geometry}
\usepackage{xcolor,fancyhdr,hyperref}
\definecolor{ink}{HTML}{17374B}
\definecolor{muted}{HTML}{53616A}
\hypersetup{colorlinks=true,linkcolor=ink,citecolor=ink,urlcolor=ink,
 pdftitle={The Modular Nash-Williams Theorem},
 pdfauthor={Amir Jafari},
 pdfsubject={Modular orientations, tree packing, nowhere-zero flows, and signed graphs}}
\newtheorem{theorem}{Theorem}
\newtheorem{lemma}[theorem]{Lemma}
\newtheorem{proposition}[theorem]{Proposition}
\newtheorem{corollary}[theorem]{Corollary}
\newcommand{\R}{\mathbb R}
\newcommand{\Z}{\mathbb Z}
\newcommand{\one}{\mathbf 1}
\newcommand{\rank}{\operatorname{rank}}
\newcommand{\im}{\operatorname{im}}
\newcommand{\dist}{\operatorname{dist}}
\begin{document}
\hypersetup{pageanchor=false}
\begin{titlepage}
\centering
\vspace*{0.65in}
{\LARGE\bfseries The Modular Nash--Williams Theorem\par}
\vspace{1.5em}
{\large Amir Jafari\\New Uzbekistan University\par}
\vspace{2em}
\begin{minipage}{0.92\textwidth}
\raggedright
\begin{center}\bfseries Abstract\end{center}
\parindent=1em
Nash--Williams's orientation theorem converts undirected edge-connectivity into directed connectivity. We extend it to prescribed outdegrees modulo $q$, whose residues must sum to the number of edges modulo $q$. For integers $q,k\ge1$, every such prescription has a realization with at least $k$ edges in each direction across each nontrivial cut if the graph is $(2q+2k-2)$-edge-connected or contains $q+2k-1$ edge-disjoint spanning trees. The method combines a cycle-space approximation bound with rounding that preserves the imbalance across every cut.

Fractional spanning-tree packing gives an exact criterion: every compatible prescription has a strongly connected realization if and only if the fractional packing number exceeds $q$. This yields a polynomial-time recognition test. With no directed connectivity requirement, $q$ disjoint spanning trees suffice, proving Hasanvand's prescribed-residue conjecture. For every odd prime $p$, this proves the conjectured statement that the union of $p$ graphic bases over the field with $p$ elements is an additive basis.

For every integer $p\ge5$, the theorem and the counterexamples of Han, Li, Wu and Zhang give the exact edge-connectivity threshold $4p+2$ for a modulo-$(2p+1)$ orientation. We also show that every six-edge-connected graph has circular flow number strictly below three, as conjectured by Li, Thomassen, Wu and Zhang. For three-flow-critical graphs on $n$ vertices and $m$ edges, we improve the previous general bound $m\le4n-10$ to $m\le3n-6$, and to $m\le3n-7$ when $n\ge5$. Further applications include Seymour's six-flow theorem, star decompositions, antisymmetric flows, and signed modular orientations.
\end{minipage}
\end{titlepage}
\setcounter{page}{1}
\hypersetup{pageanchor=true}

\section{Introduction}

Nash--Williams's orientation theorem~\cite{NW} says that a $2k$-edge-connected graph can be oriented so that at least $k$ edges cross every nontrivial cut in each direction. Now prescribe the outdegree of each vertex modulo $q$. There is an immediate compatibility condition: the prescribed residues must sum to the number of edges modulo $q$. Even without a directed connectivity requirement, finding an orientation that meets these local arithmetic constraints is a central problem in flow theory. We ask how much undirected connectivity suffices to meet the prescriptions while retaining directed connectivity.

Throughout the unsigned part of the paper, $G=(V,E)$ is a finite loopless multigraph with $n=|V|\ge2$ vertices and $m=|E|$ edges; parallel edges are distinct. We write $\Z_q=\Z/q\Z$ for the residues modulo a positive integer $q$. For $\varnothing\ne S\subsetneq V$, let $\delta_G(S)$ be the edges with exactly one endpoint in $S$, and put $\lambda(G)=\min_{\varnothing\ne S\subsetneq V}|\delta_G(S)|$. These are the \emph{nontrivial cuts}; for an integer $t\ge1$, $G$ is $t$-edge-connected when $\lambda(G)\ge t$.

In an orientation $D$, let $d_D^+(S)$ and $d_D^-(S)$ count the edges leaving and entering $S$. Write $d_D^\pm(v)=d_D^\pm(\{v\})$ and $\deg_G(v)=d_D^+(v)+d_D^-(v)$; the notation $\pm$ asserts both sign choices. The orientation is \emph{$k$-arc-connected} when both cut counts are at least $k$, and \emph{strongly connected} when $k=1$, equivalently when every ordered pair of vertices is joined by a directed path.

\begin{samepage}
A function $\pi:V\to\Z_q$ is \emph{compatible} if
\begin{equation}\label{eq:compatible}
 \sum_{v\in V}\pi(v)\equiv m\pmod q,
\end{equation}
the necessary condition obtained by summing all outdegrees.
\end{samepage}

\begin{theorem}[Modular Nash--Williams]\label{thm:main}
Let $q,k\ge1$ be integers. Suppose that either
\begin{enumerate}
\item $G$ is $(2q+2k-2)$-edge-connected; or
\item $G$ contains $q+2k-1$ pairwise edge-disjoint spanning trees.
\end{enumerate}
Then every compatible function $\pi:V\to\Z/q\Z$ is realized by a $k$-arc-connected orientation $D$ satisfying
\[
 d_D^+(v)\equiv\pi(v)\pmod q\qquad(v\in V).
\]
\end{theorem}

At $q=1$ the congruence disappears and the connectivity hypothesis is $2k$, recovering the uniform form of Nash--Williams's theorem. For general $q$, another $2q-2$ units of undirected edge-connectivity suffice to impose arbitrary compatible residues. At $q=2$, Frank and Kir\'aly~\cite[Theorem~4.1]{FK} give an exact criterion: every compatible parity prescription has a $k$-arc-connected realization if and only if every partition into $r\ge2$ parts has at least $(k+1)r-1$ edges joining distinct parts. Our theorem gives a uniform sufficient condition for all moduli.

\paragraph{Fractional packing and an exact criterion.}
The parameter behind both hypotheses is the fractional spanning-tree packing number $\tau_f(G)$. Assign nonnegative weights to spanning trees, with the sum of weights of the trees through any edge at most one. Then $\tau_f(G)$ is the maximum total weight. If $\tau_f(G)>q$, our fractional theorem gives a realizing orientation with
\[
 d_D^\pm(S)\ge
 \left\lceil\frac{|\delta_G(S)|}{2}
       \left(1-\frac{q}{\tau_f(G)}\right)\right\rceil
 \qquad(\varnothing\ne S\subsetneq V).
\]
In fact, every compatible prescription has a strongly connected realization exactly when $\tau_f(G)>q$. Equivalently, deleting any edge leaves $q$ disjoint spanning trees. This yields a polynomial-time test for the universal property. The case $q=2$ recovers the strong-connectivity case of Frank and Kir\'aly's theorem on parity orientations~\cite{FK}. Our construction leaves open the efficient computation of a realizing orientation for an individual prescription.

At the endpoint $\tau_f(G)\ge q$, one can still realize every compatible prescription. For $q\ge3$, the consequence for $q$ edge-disjoint spanning trees proves Hasanvand's prescribed-residue conjecture~\cite[Conjecture~3]{Hasanvand2016}. For every odd prime $p$, it also proves the $p$-basis form of the graphic additive basis conjecture~\cite{JLPT,CLSWWei}. Here an additive basis is a multiset of vectors whose subset sums cover the whole space. Start from a reference orientation: reversing an edge transfers one unit of outdegree between its endpoints, so the changes available by reversing a subset are the $0$--$1$ sums of oriented incidence columns. Over the field with $p$ elements, the columns of a spanning tree form a basis of the vertex vectors whose coordinates sum to zero. Thus the union of $p$ such graphic bases is an additive basis. Dimension-independent bounds on the number of bases were already known, even for the broader class of vectors supported on at most two coordinates~\cite{EJLT}; the conclusion here is the conjectured count $p$.

\paragraph{Consequences for flows.}
A modulo-$(2p+1)$ orientation has outdegree congruent to indegree modulo $2p+1$ at every vertex. Jaeger asked whether $4p$-edge-connectivity suffices for such orientations~\cite{Jaeger88}. Lov\'asz, Thomassen, Wu and Zhang proved that $6p$ suffices~\cite{LTWZ}. Han, Li, Wu and Zhang constructed $(4p+1)$-edge-connected counterexamples for $p\ge5$~\cite[Theorem~1.4]{HLWZ}. Theorem~\ref{thm:main} with $q=2p+1$ and $k=1$ gives the matching upper bound. Thus $4p+2$ is the exact universal edge-connectivity threshold for modulo-$(2p+1)$ orientations when $p\ge5$. For modulus five our bound is ten. Jaeger's reduction shows that a guarantee at nine would imply Tutte's five-flow conjecture~\cite[Proposition~4.2]{LTWZ}.

A circular $r$-flow assigns signed real values to the edges, with zero net outflow at every vertex and magnitudes between $1$ and $r-1$. The circular flow number $F_c(G)$ is the infimum of the feasible $r\ge2$. We prove $F_c(G)\le 2+2/(\tau_f(G)-1)$ for every connected bridgeless graph. In particular, $(4p+2)$-edge-connectivity gives $F_c(G)<2+1/p$. At $p=1$ this resolves the conjecture of Li, Thomassen, Wu and Zhang that every six-edge-connected graph has circular flow number strictly less than three~\cite[Conjecture~1.3]{LTWZindex}.

A \emph{balanced boundary} modulo $q$ assigns residues summing to zero to the vertices; a realizing orientation has outdegree minus indegree equal to those residues modulo $q$. The same argument yields Seymour's six-flow theorem, the star case of Bar\'at--Gerbner's packing-decomposition conjecture, and antisymmetric $\Z_5$-flows in every orientation of every ten-edge-connected graph. It also strengthens Chen and Ning's theorem on five-edge-connected graphs with at most five cuts of size five: every balanced boundary modulo three has a realization that is directed strongly connected~\cite{ChenNing}. For signed graphs we develop a rounding argument using inequalities for subgraphs with prescribed degrees. It gives a sufficient connectivity bound of $2k+2$ for modulo-$k$ orientations under the essential unbalance hypothesis, improving $6k-10$ for odd $k\ge5$~\cite{LMTZ}. Section~\ref{sec:signed} gives the signed-graph definitions and precise hypotheses.

\paragraph{Density of critical graphs.}
A connected bridgeless graph is three-flow-critical if it has no nowhere-zero integer three-flow but every edge contraction has one. Li, Ma, Shi, Wang and Wu proved $m\le4n-10$~\cite{LMSWW}. The recent bound $m\le3n-5$ of~\cite[Theorem~1.6]{ADLMS} assumes at most one vertex of degree at least seven and criticality under arbitrary vertex identifications. We prove $m\le3n-6$ using the cycle-space approximation bound and a parity argument, and sharpen it to $m\le3n-7$ for $n\ge5$ in Appendix~\ref{app:critical-density}. These bounds strengthen the general $3n-5$ conjecture stated in~\cite[Conjecture~1.3]{DMdensity} and come within one edge of the sharp conjectured bound $3n-8$ for $n\ge7$ in~\cite[Conjecture~1.3]{LMSWW}.

The density proof uses parity: an integral boundary with the parity of the vertex degrees can violate a cut inequality only by at least two. On a minimal subgraph violating $m\le3n-6$, the approximation error is smaller than two on every cut with connected shores. The resulting orientations contradict criticality. At the tight count $m=3n-6$, equality forces both shores of a violating cut to be trees. A perturbation of the edgewise error bounds handles the degree-five obstruction, and local reductions complete the proof for $n\ge5$.

\paragraph{Why approximation controls all cuts.}
Start each edge halfway between its two possible directions. A shift in the real cycle space preserves the vertex boundary, while an integral edge shift multiplied by $q$ preserves the target residues. A simultaneous approximation lemma, based on the method of Alon, Buci\'c and Davies~\cite{ABD}, lets us combine these shifts to obtain a fractional orientation close to the halfway point. The error bound is $q/(2\tau_f(G))$ per edge. The link is a rank identity: the rank of the cycle-space coordinates on an edge set equals its size minus the increase in the number of components when those edges are deleted. The partition formula for fractional tree packing turns these rank losses into the approximation bound. Finally, rounding along cycles preserves the exact vertex boundary. Summing over a shore cancels the contributions of its internal edges, so rounding also preserves the imbalance across every cut. This is why one approximation gives the directed cut bounds simultaneously.

\paragraph{Relation to the earlier preprint.}
The earlier preprint \emph{Graphic additive bases and graph flows}~\cite{Jafari} already gives the prescribed-residue theorem (Theorem~1.1), both unsigned arc-connectivity hypotheses and the quantitative cut estimate (Proposition~4.1 and its proof), the exact strong-realization criterion (Corollary~4.2), and the strict circular-flow and exact large odd-modulus thresholds (Corollary~1.3). The present article collects those results in a common orientation framework, gives a self-contained rank-ratio approximation proof, extends the rounding method to signed graphs, and proves the new density bounds for three-flow-critical graphs. This article supersedes that preprint, which will not be submitted separately.

ChatGPT, including Codex, assisted with developing and checking proof arguments, particularly the density arguments in Section~\ref{sec:critical-density} and Appendix~\ref{app:critical-density}; the AI methodology statement at the end describes its use.

\section{The approximation lemma}\label{sec:approx}

We use $\R$ and $\Z$ for the real numbers and integers. Vectors are columns, $A^{\mathsf T}$ denotes the transpose of $A$, and all ranks, spans, images and kernels are over $\R$ unless a modulus is specified. For a positive integer $r$, write $[r]=\{1,\ldots,r\}$. The vector $\one$ has all coordinates equal to one in the indicated space. Vector inequalities and floor operations are coordinatewise.

We write $\dist(t,\Z)=\min_{z\in\Z}|t-z|$. The following result of Alon, Buci\'c and Davies gives a short route to the approximation lemma. Appendix~\ref{app:approx} proves the approximation lemma directly.

\begin{samepage}
\begin{lemma}[ABD, Lemma~7~\cite{ABD}]\label{lem:ABD}
Let $K,d\ge1$ be integers and let $A_1,\ldots,A_K$ be nonsingular integer $d\times d$ matrices. For any $c_1,\ldots,c_K\in\R^d$, there exists $y\in\R^d$ such that
\[
 \dist\bigl((A_\ell y-c_\ell)_j,\Z\bigr)\ge\frac1{2K}
 \qquad(1\le\ell\le K,\ 1\le j\le d).
\]
\end{lemma}
\end{samepage}

We also use Edmonds's matroid partition theorem~\cite{Edmonds}: a finite indexed family of vectors can be partitioned into $K$ linearly independent families if and only if every subfamily $X$ satisfies $|X|\le K\rank(X)$. Repeated vectors are allowed as separate indexed elements.

The following consequence is the only approximation statement needed for the modular orientation theorem. For any matrix $A$ whose rows have an index set $I$ and any $F\subseteq I$, let $A_F$ be its row restriction to $F$.

\begin{samepage}
\begin{lemma}[Shifted approximation]\label{lem:approx}
Let $A\in\Z^{m\times d}$ have $d\ge1$ independent columns and no zero rows. For $v\in\R^m$, write $\|v\|_\infty=\max_{i\in[m]}|v_i|$ for the infinity norm. Put
\[
 \mu=\min_{\varnothing\ne F\subseteq\{1,\ldots,m\}}
             \frac{\rank(A_F)}{|F|}.
\]
For every $b\in\R^m$, there exist $x\in\R^d$ and $z\in\Z^m$ such that
\begin{equation}\label{eq:approx}
 \|Ax-b-z\|_\infty\le\frac{1-\mu}{2}.
\end{equation}
\end{lemma}
\end{samepage}

\begin{proof}
The number $\mu$ is rational and lies in $(0,1]$; write $\mu=a/c$ with positive integers $a\le c$. Make $a$ labelled copies of every row of $A$. If $X$ is a subfamily of these copies and $F$ is its set of original row indices, then
\[
 |X|\le a|F|\le c\rank(A_F)=c\rank(X).
\]
Edmonds's theorem partitions the copies into $c$ independent families. Complete each family to a basis of $\R^d$ using integer row vectors, giving $c$ nonsingular integer matrices.

Write $u_i^{\mathsf T}$ for row $i$ of $A$. To its copy labelled $j\in\{0,\ldots,a-1\}$ assign the shift
\[
 \frac{b_i-\tfrac12+j}{a}.
\]
Assign arbitrary shifts to the added rows. Lemma~\ref{lem:ABD} gives a common $y\in\R^d$ for which every labelled row satisfies
\[
 \dist\left(u_i^{\mathsf T}y-\frac{b_i-\tfrac12+j}{a},\Z\right)
 \ge\frac1{2c}.
\]
The elementary identity
\[
 \min_{0\le j<a}\dist(t-j/a,\Z)=\frac1a\dist(at,\Z)
\]
therefore yields
\[
 \dist\bigl(a u_i^{\mathsf T}y-b_i+\tfrac12,\Z\bigr)
 \ge\frac{a}{2c}=\frac\mu2.
\]
Since $\dist(t+\tfrac12,\Z)=\tfrac12-\dist(t,\Z)$, setting $x=ay$ and choosing a nearest integer $z_i$ to $u_i^{\mathsf T}x-b_i$ proves~\eqref{eq:approx}.
\end{proof}

\section{Fractional packing and the cycle space}\label{sec:packing}

The fractional spanning-tree packing number of a connected graph is also called its strength and is computable in polynomial time~\cite{Cunningham}.

A \emph{fractional spanning-tree packing} assigns nonnegative weights $w_T$ to the spanning trees of $G$, subject to $\sum_{T\ni e}w_T\le1$ for every edge $e$. Its maximum total weight is denoted by $\tau_f(G)$. For a disconnected graph, set $\tau_f(G)=0$. Write $\tau_{\rm int}(G)$ for the maximum number of pairwise edge-disjoint spanning trees.

For a partition $\mathcal P$ of $V$ into nonempty parts, let $e_G(\mathcal P)$ count the edges joining different parts. The Nash--Williams--Tutte tree-packing theorem~\cite{NWtrees,Tutte} states that $G$ contains $t$ edge-disjoint spanning trees precisely when
\[
 e_G(\mathcal P)\ge t(|\mathcal P|-1)
 \quad\hbox{for every partition }\mathcal P.
\]
Consequently,
\begin{equation}\label{eq:packing}
 \tau_f(G)=\min_{|\mathcal P|\ge2}
           \frac{e_G(\mathcal P)}{|\mathcal P|-1},
 \qquad
 \tau_{\rm int}(G)=\lfloor\tau_f(G)\rfloor.
\end{equation}
Here is a direct justification of the fractional equality. Every spanning tree crosses a partition into $r$ parts at least $r-1$ times, so the partition minimum bounds every fractional packing from above. That minimum is rational; call it $s$ and choose a positive integer $L$ with $Ls$ integral. Replacing each edge by $L$ parallel copies gives a graph satisfying the partition inequalities for $Ls$ trees. Pack those trees and project them to $G$, assigning each weight $1/L$. Each original edge has total load at most one, and the total weight is $s$.

Two useful comparisons, for connected $G$, are
\begin{equation}\label{eq:lambda-tau}
 \frac{n\lambda(G)}{2(n-1)}\le\tau_f(G)\le\lambda(G).
\end{equation}
For the upper bound use a partition into two parts defining a minimum cut. For the lower bound, sum the cut sizes of the $r$ parts of any partition:
\[
 2e_G(\mathcal P)\ge r\lambda(G),
 \qquad
 \frac{r}{r-1}\ge\frac{n}{n-1}.
\]

We now express the same parameter as a rank condition. For connected $G$, fix a reference orientation and let $B\in\Z^{V\times E}$ be its incidence matrix, with $+1$ at the tail of an edge and $-1$ at its head. For an edge vector $f$, the vector $Bf$ is its \emph{boundary}, recording net outflow at each vertex. Thus a real circulation satisfies $Bf=0$, and the real \emph{cycle space} is $\ker B$. Fix a spanning tree. Each edge outside it forms a unique cycle with the tree path joining its endpoints. Choosing a direction around each such \emph{fundamental cycle} gives a column with entries $+1$ or $-1$ according to agreement with the reference orientation, and zero off the cycle. These columns form an integer matrix
\[
 U\in\Z^{E\times d},\qquad d=m-n+1,
 \qquad BU=0,\quad\im U=\ker B,
\]
whose columns are independent. For $F\subseteq E$, let $G-F$ be the spanning subgraph obtained by deleting $F$, and let $c(G-F)$ count its components, including isolated vertices. We abbreviate $G-\{e\}$ to $G-e$.

\begin{lemma}[Cycle-space rank]\label{lem:rank}
If $G$ is connected and $\tau=\tau_f(G)>1$, then $d\ge1$, $U$ has no zero rows, and
\begin{equation}\label{eq:rank}
 \rank(U_F)=|F|-c(G-F)+1,
 \qquad
 \min_{\varnothing\ne F\subseteq E}
       \frac{\rank(U_F)}{|F|}=1-\frac1\tau.
\end{equation}
\end{lemma}

\begin{proof}
The singleton partition gives $\tau\le m/(n-1)$, so $\tau>1$ implies $m\ge n$ and $d\ge1$. The kernel of the restriction of the cycle space to $F$ consists exactly of the circulations supported on $G-F$. Its dimension is
\[
 (m-|F|)-n+c(G-F).
\]
Subtracting this from $d$ proves the first equality in~\eqref{eq:rank}.

The partition formula also gives
\begin{equation}\label{eq:deletions}
 \tau=\min_{F:\,c(G-F)\ge2}\frac{|F|}{c(G-F)-1}.
\end{equation}
Indeed, the components of $G-F$ form a partition whose crossing edges all lie in $F$, so every displayed ratio is at least $\tau$. Conversely, take a minimizing partition with $r$ parts and delete all its crossing edges. Then $c(G-F)\ge r$, so the resulting ratio is at most $\tau$.

Combine~\eqref{eq:deletions} with the first equality in~\eqref{eq:rank}. Sets $F$ for which $G-F$ remains connected have rank ratio one; all others have minimum ratio $1-1/\tau$. This is positive, so in particular every row of $U$ is nonzero.
\end{proof}

Under the hypothesis $\tau=\tau_f(G)>1$, applying Lemma~\ref{lem:approx} to $U$ gives the precise estimate needed below: for every $b\in\R^E$, there are $x\in\R^d$ and $z\in\Z^E$ with
\begin{equation}\label{eq:cycle-approx}
 \|Ux-b-z\|_\infty\le\frac1{2\tau}.
\end{equation}

\section{From approximation to an orientation}\label{sec:orientation}

We first record the elementary rounding fact that converts a fractional choice of directions into an orientation.

\begin{lemma}[Rounding with fixed boundary]\label{lem:round}
If $X\in[0,1]^E$ and $BX\in\Z^V$, there exists $y\in\{0,1\}^E$ such that $By=BX$. Coordinates already equal to $0$ or $1$ can be kept fixed.
\end{lemma}

\begin{proof}
Consider the edges for which $X_e\notin\{0,1\}$. If this subgraph has an edge but no cycle, it has a leaf. At that vertex all but one incident contribution to $BX$ are integral, contradicting $BX\in\Z^V$. Thus, unless $X$ is already integral, the fractional edges contain a cycle. Choose a direction around it and let $\gamma\in\{-1,0,1\}^E$ be its oriented cycle vector: entries are $+1$ for agreement with the reference direction, $-1$ for disagreement, and zero off the cycle. Then $B\gamma=0$. Increase $t\ge0$ in $X+t\gamma$ until a coordinate first reaches $0$ or $1$. All coordinates remain in $[0,1]$, the boundary is unchanged, and the number of fractional edges decreases. Repeating finishes in at most $m$ steps. A pair of parallel edges may form a cycle of length two.
\end{proof}

\begin{theorem}[Modular Nash--Williams, fractional form]\label{thm:fractional}
Let $q\ge1$ be an integer, let $G$ be connected, and suppose $\tau=\tau_f(G)>q$. Every compatible prescription $\pi$ has an orientation $D$ such that
\begin{equation}\label{eq:quantitative}
 d_D^+(v)\equiv\pi(v)\pmod q,
 \qquad
 d_D^\pm(S)\ge
 \left\lceil\frac{|\delta_G(S)|}{2}
                  \left(1-\frac q\tau\right)\right\rceil
\end{equation}
for every vertex $v$ and every nonempty proper set $S$. In particular, for an integer $k\ge1$, the orientation is $k$-arc-connected whenever
\begin{equation}\label{eq:master}
 \lambda(G)\left(1-\frac q{\tau_f(G)}\right)>2k-2.
\end{equation}
\end{theorem}

\begin{proof}
Let $d_0^-(v)$ be the indegree in the reference orientation. Choose an integer vector $h\in\Z^V$ with
\[
 h_v\equiv\pi(v)-d_0^-(v)\pmod q,
 \qquad \sum_v h_v=0.
\]
Such a choice exists: compatibility makes the sum of any integer representatives a multiple of $q$, which can be subtracted at a single vertex.

Since $G$ is connected, there is an integer vector $a\in\Z^E$ with $Ba=h$. To see this, fix a root and route the integer amount $h_v$ from each nonroot vertex $v$ to the root along a spanning tree, using signs relative to the reference orientation. The resulting boundary is $h$ at every vertex.

Apply~\eqref{eq:cycle-approx} with
\[
 b=\frac{\tfrac12\one-a}{q}.
\]
Choose $x,z$ accordingly and set
\begin{equation}\label{eq:fractional-direction}
 X=\tfrac12\one+q(Ux-b-z)=a+qUx-qz.
\end{equation}
The approximation bound gives
\begin{equation}\label{eq:central}
 \left|X_e-\tfrac12\right|\le\frac q{2\tau}
 \qquad(e\in E).
\end{equation}
Since $\tau>q$, this places $X$ in the open cube $(0,1)^E$. Also $BU=0$, so
\[
 BX=h-qBz\in\Z^V,\qquad BX\equiv h\pmod q.
\]
Lemma~\ref{lem:round} gives $y\in\{0,1\}^E$ with $By=BX$. Orient an edge as in the reference orientation when $y_e=1$, and oppositely when $y_e=0$. Starting from the reversed reference orientation, switching an edge back to its reference direction adds its column of $B$ to the outdegree vector. Hence
\[
 d_D^+(v)=d_0^-(v)+(By)_v
          \equiv d_0^-(v)+h_v\equiv\pi(v)\pmod q.
\]

It remains to check the directed cuts. Fix $S$ and put $C=\delta_G(S)$. For $e\in C$, set $\sigma_e=1$ if its reference direction leaves $S$, and $\sigma_e=-1$ otherwise. Summing $B(y-X)=0$ over $S$ gives
\[
 \sum_{e\in C}\sigma_e(y_e-X_e)=0.
\]
Therefore~\eqref{eq:central} implies
\begin{align*}
 \bigl|d_D^+(S)-d_D^-(S)\bigr|
 &=\left|\sum_{e\in C}\sigma_e(2y_e-1)\right|\\
 &=\left|\sum_{e\in C}\sigma_e(2X_e-1)\right|
 \le\frac q\tau|C|.
\end{align*}
Since $d_D^+(S)+d_D^-(S)=|C|$, each directed count is at least $|C|(1-q/\tau)/2$. The counts are integers, proving~\eqref{eq:quantitative}. Under~\eqref{eq:master}, this lower bound exceeds $k-1$ on every nontrivial cut, so both counts are at least $k$.
\end{proof}

\section{Recovering the two hypotheses}\label{sec:hypotheses}

We can now prove Theorem~\ref{thm:main} and state its connection with fractional packing explicitly. Write $\lambda=\lambda(G)$ and $\tau=\tau_f(G)$.

\paragraph{Edge-connectivity.}
If $\lambda\ge2q+2k-2$, inequality~\eqref{eq:lambda-tau} gives
\[
 \lambda\left(1-\frac q\tau\right)
 \ge\lambda-\frac{2q(n-1)}n
 \ge2k-2+\frac{2q}n>2k-2.
\]
In particular $\tau>q$, and Theorem~\ref{thm:fractional} supplies the required orientation. The term $2q/n$ explains why the non-strict connectivity assumption yields the strict inequality needed for the final integer rounding of the cut bound.

\paragraph{Edge-disjoint spanning trees.}
If $G$ contains $t=q+2k-1$ edge-disjoint spanning trees, then $\tau\ge t$. Also $\lambda\ge t$, because every tree meets every nontrivial cut. Hence
\[
 \lambda\left(1-\frac q\tau\right)\ge t-q=2k-1>2k-2,
\]
and the same theorem applies. This completes the proof of Theorem~\ref{thm:main}.

\paragraph{A condition using only fractional packing.}
By~\eqref{eq:packing}, the tree-packing hypothesis is exactly $\tau_f(G)\ge q+2k-1$. The argument also gives the following sufficient condition with a real threshold:
\begin{corollary}\label{cor:tau}
Let $q,k\ge1$ be integers. If $\tau_f(G)>q+2k-2$, every compatible prescription modulo $q$ is realized by a $k$-arc-connected orientation.
\end{corollary}
\begin{proof}
Since $\lambda\ge\tau>q$, we have
\[
 \lambda\left(1-\frac q\tau\right)\ge\tau-q>2k-2.
\]
Apply Theorem~\ref{thm:fractional}.
\end{proof}

\section{Prescribed residues: Hasanvand's conjecture}\label{sec:residues}

A \emph{balanced boundary} modulo $q$ is a function $\beta:V\to\Z/q\Z$ with $\sum_v\beta(v)=0$. For odd $q\ge3$, prescribing
\[
 d_D^+(v)-d_D^-(v)\equiv\beta(v)\pmod q
\]
is equivalent to prescribing outdegrees $\pi(v)=(\deg_G(v)+\beta(v))/2$ modulo $q$, where division by two means multiplication by its inverse in $\Z_q$; this prescription is compatible. A graph is \emph{strongly $\Z_q$-connected} if every balanced boundary has such a realizing orientation. This terminology does not require that the realizing orientation be directed strongly connected.

\begin{theorem}[Prescribed residues]\label{thm:ordinary-residues}
Let $q\ge2$ be an integer. If $\tau_f(G)\ge q$, every compatible prescription modulo $q$ has a realizing orientation. In particular, $q$ edge-disjoint spanning trees suffice. For $q=2$, connectedness alone suffices. For odd $q$, these hypotheses imply strong $\Z_q$-connectivity.
\end{theorem}
\begin{proof}
Put $\tau=\tau_f(G)\ge q\ge2$. The construction through~\eqref{eq:fractional-direction} only requires $\tau>1$. Its error bound~\eqref{eq:central} gives $|X_e-1/2|\le q/(2\tau)\le1/2$, so $X\in[0,1]^E$ even when $\tau=q$. The boundary $BX$ is integral and encodes the prescribed residues after adding the reference indegrees. Lemma~\ref{lem:round} gives a binary vector with the same boundary, which determines the required orientation. The tree conclusion follows from~\eqref{eq:packing}. At $q=2$, let $d_0^-$ be the reference indegree vector and choose $h\in\Z^V$ with $\sum_v h_v=0$ and $h\equiv\pi-d_0^-\pmod2$. Route $h$ along a spanning tree and reduce each integer edge coefficient modulo two. The resulting binary vector already has the required boundary modulo two. The odd-modulus conclusion follows from the preceding conversion of boundaries to outdegrees.
\end{proof}

For $q\ge3$, the tree conclusion is Hasanvand's prescribed-residue conjecture~\cite[Conjecture~3]{Hasanvand2016}. For residue realization alone, it reduces the earlier sufficient tree count $2q-2$ of Hasanvand~\cite[Theorem~1.6]{Hasanvand2024} to $q$; his theorem additionally controls the outdegrees. At modulus three, it reduces the four-tree bound of Han, Lai and Li~\cite{HLL} to three.

\section{Strongly connected realizations and small cuts}\label{sec:strong}

The fractional criterion first gives an exact characterization of directed strongly connected realizations. At $q=1$ this is the bridgeless orientation theorem. At $q=2$ it is the strong-connectivity case of the parity-orientation theorem of Frank and Kir\'aly~\cite{FK}. The proof below gives a single argument for all moduli.

\begin{theorem}[Strongly connected realizations]\label{thm:strong-exact}
Let $q\ge1$ be an integer and let $G$ be connected. The following are equivalent.
\begin{enumerate}
\item Every compatible prescription modulo $q$ has a strongly connected realization.
\item $e_G(\mathcal P)>q(|\mathcal P|-1)$ for every vertex partition $\mathcal P$ with at least two parts.
\item $\tau_f(G)>q$.
\item $G-e$ contains $q$ edge-disjoint spanning trees for every edge $e$.
\end{enumerate}
\end{theorem}
\begin{proof}
The partition formula~\eqref{eq:packing} gives (2)$\Leftrightarrow$(3), and Theorem~\ref{thm:fractional} gives (3)$\Rightarrow$(1). To see that (1) implies (2), suppose a partition $\mathcal P=\{V_1,\ldots,V_r\}$ into $r\ge2$ parts has $M\le q(r-1)$ crossing edges. Contracting each part gives a connected graph on $r$ vertices with $M$ non-loop edges, so $M\ge r-1$, hence $r\le M+q\le rq$. We can therefore choose integers $a_i\in\{1,\ldots,q\}$ with $\sum_i a_i=M+q$. Let $E_G(V_i)$ denote the edges with both endpoints in $V_i$. Prescribe vertex residues whose sum on $V_i$ is $|E_G(V_i)|+a_i$ modulo $q$. Their total is compatible. A strongly connected realization has positive outgoing cut count at each part, congruent to $a_i$, hence at least $a_i$. Summing over the parts would give $M\ge M+q$, a contradiction.

For the last equivalence, $\tau_f>q$ means that each $r$-part partition has at least $q(r-1)+1$ crossing edges, so one deletion leaves the packing inequalities. Conversely, delete a crossing edge of any fixed partition and apply those inequalities.
\end{proof}

\paragraph{Recognition.}
The edge-deletion criterion is high $q$-tree-connectivity, for which Kir\'aly~\cite{Kiraly} gives a dedicated recognition algorithm. For completeness, one can also test the criterion by matroid union. If $q(n-1)>m-1$, reject immediately. Otherwise, for each edge $e$ use matroid union~\cite{Edmonds} to compute the maximum number of edges of $G-e$ that can be partitioned into $q$ forests. This maximum is $q(n-1)$ exactly when there are $q$ edge-disjoint spanning trees, since each of the $q$ forests must then have $n-1$ edges. The preliminary check bounds the number of forests polynomially even when $q$ is given in binary. This algorithm recognizes the universal property. Once a fractional witness is available, Lemma~\ref{lem:round} constructs its orientation in at most $m$ cycle roundings; no polynomial-time method for finding that witness is established here. Appendix~\ref{app:approx} explains what the exchange argument does construct for fixed weights. Frank, Jord\'an and Szigeti~\cite{FJS} posed the separate problem of deciding whether a single compatible parity prescription has a strongly connected realization.

Chen and Ning~\cite{ChenNing} proved that the following hypothesis guarantees $\Z_3$-connectivity. The fractional criterion also gives directed strong connectivity for the realizing orientations.

\begin{theorem}[Five cuts of size five]\label{thm:five-cuts}
Every five-edge-connected graph with at most five distinct cuts of size five is strongly $\Z_3$-connected. Every balanced boundary can in fact be realized by a directed strongly connected orientation.
\end{theorem}
\begin{proof}
For a two-part partition, the crossing count is at least five and hence greater than three. For a partition $\mathcal P=\{V_1,\ldots,V_r\}$ with $r\ge3$, the cuts $\delta_G(V_i)$ are distinct: equality for two parts would give $\delta_G(V_i\cup V_j)=\varnothing$, contradicting connectedness. At most five have size five, and every other one has size at least six. Thus
\[
 2e_G(\mathcal P)=\sum_i|\delta_G(V_i)|\ge6r-5,
 \qquad e_G(\mathcal P)\ge3r-2>3(r-1).
\]
It follows that $\tau_f(G)>3$. Apply Theorem~\ref{thm:strong-exact} and the odd-modulus boundary conversion. Cuts are counted as edge sets, so complementary shores count once.
\end{proof}

\section{Star decompositions}\label{sec:stars}

The star $K_{1,q}$ has one centre joined to $q$ leaves. A decomposition into such stars partitions the edge set into subgraphs isomorphic to $K_{1,q}$.

\begin{theorem}[Star decompositions]\label{thm:stars}
Let $q\ge2$ be an integer. If a simple graph has $q$ edge-disjoint spanning trees and $q\mid |E|$, it decomposes into copies of $K_{1,q}$. For $q=2$, connectedness and even $|E|$ suffice.
\end{theorem}
\begin{proof}
Prescribe outdegree zero modulo $q$ using Theorem~\ref{thm:ordinary-residues}, and group each vertex's outgoing edges into sets of $q$. Each group is a star because the graph is simple, and every edge belongs to exactly one group.
\end{proof}
This proves the star case of Bar\'at--Gerbner's packing-decomposition conjecture~\cite[Conjecture~6]{BG}.

\section{Group connectivity and integer flows}\label{sec:flows}

Fix a reference orientation with incidence matrix $B$. A \emph{flow} is an edge vector $f$ satisfying conservation, $Bf=0$. Its values may be real, integer, or elements of $\Z_q=\Z/q\Z$; it is \emph{nowhere-zero} when every edge value is nonzero. A graph is \emph{$\Z_q$-connected} if every balanced boundary $\beta$ has a solution
\[
 Bf=\beta\quad\hbox{in }\Z_q^V,
 \qquad f_e\ne0\quad(e\in E).
\]
This property is independent of the reference orientation. For odd $q\ge3$, strong $\Z_q$-connectivity requires the values to lie in $\{1,-1\}$; reversing the edges with value $-1$ gives the realizing orientation. At $q=3$ these two notions coincide.

\begin{theorem}[Integer intervals]\label{thm:interval}
Let $q\ge2$ and $1\le L<q$ be integers. If $L\tau_f(G)\ge q$, then, for any $r\in\Z^E$ and every balanced boundary $\beta:V\to\Z_q$, there is an integer vector $f\in\Z^E$ satisfying
\[
 Bf\equiv\beta\pmod q,\qquad
 r_e\le f_e\le r_e+L\qquad(e\in E).
\]
\end{theorem}
\begin{proof}
Put $\tau=\tau_f(G)>1$. Route an integer lift of the boundary to obtain $a\in\Z^E$ with $Ba\equiv\beta\pmod q$. Set $c=r+(L/2)\one$ and apply~\eqref{eq:cycle-approx} with $b=(c-a)/q$. The resulting vector
\[
 X=c+q(Ux-b-z)=a+qUx-qz
\]
satisfies $|X_e-c_e|\le q/(2\tau)\le L/2$, while $BX=Ba-qBz$ is integral and has the required residues. Apply Lemma~\ref{lem:round} to $X-\lfloor X\rfloor$ and add back $\lfloor X\rfloor$. The rounding fixes integral coordinates. Each other coordinate is rounded to its floor or ceiling, both inside the specified interval because its endpoints are integers. The boundary is preserved.
\end{proof}

An \emph{integer $q$-flow} is an integer circulation with $1\le|f_e|\le q-1$. Its equivalence with a nowhere-zero $\Z_q$-flow is Tutte's modular-flow theorem~\cite[Theorem~3 of the arXiv version]{DFRS}. Here is a direct rounding proof. Given a nowhere-zero $\Z_q$-flow, choose representatives $g_e\in\{1,\ldots,q-1\}$, round $g/q$ by Lemma~\ref{lem:round} to $y\in\{0,1\}^E$ with $By=B(g/q)$, and take $f=g-qy$. Conversely, reducing an integer $q$-flow modulo $q$ gives a nowhere-zero modular flow.

\begin{theorem}[Group connectivity and integer flows]\label{thm:nonzero-connectivity}
For every integer $q\ge3$, if $\tau_f(G)\ge q/(q-2)$, then $G$ is $\Z_q$-connected and has a nowhere-zero integer $q$-flow.
\end{theorem}
\begin{proof}
Use Theorem~\ref{thm:interval} with intervals $[1,q-1]$. For zero boundary, lift the modular flow as above.
\end{proof}

\section{The density of three-flow-critical graphs}\label{sec:critical-density}

All graphs in this section are loopless; parallel edges are allowed unless
simplicity is stated. For an edge $e$, the contraction $G/e$ identifies its
endpoints and deletes the resulting loops. A connected bridgeless graph
is \emph{three-flow-critical} if it has no nowhere-zero integer $3$-flow,
but $G/e$ has one for every edge $e$. By the conversion in
Section~\ref{sec:flows}, we may equivalently use zero-boundary orientations
modulo three. The bridgeless convention excludes $K_2$. If bridges are
allowed, $K_2$ is the only additional critical graph: contracting any other
edge would preserve a bridge and prevent a nowhere-zero flow.

A graph is \emph{$\Z_3$-reduced} (also called
\emph{$\Z_3$-irreducible}~\cite{LMSWW}) if none of its subgraphs on at least two vertices is
strongly $\Z_3$-connected. Write $e_H(Y)$ for the number of edges of $H$
with both endpoints in $Y$. A simple graph $H$ is \emph{$(3,6)$-sparse} if
$e_H(Y)\le3|Y|-6$ for every $Y\subseteq V(H)$ with $|Y|\ge3$; it is \emph{$(3,6)$-tight} if also
$|E(H)|=3|V(H)|-6$.

\begin{theorem}[Reduced density]\label{thm:reduced-density}\label{thm:critical-density}
Every $\Z_3$-reduced graph is simple and $(3,6)$-sparse. Consequently every
three-flow-critical graph on $n$ vertices has at most $3n-6$ edges.
\end{theorem}

The previous general upper bound is $4n-10$~\cite[Theorem~1.2]{LMSWW}.
The recent $3n-5$ bound of~\cite[Theorem~1.6]{ADLMS} assumes at most one
vertex of degree at least seven and criticality under arbitrary vertex
identifications. The theorem above applies to every three-flow-critical
graph under the edge-contraction definition given here.
Appendix~\ref{app:critical-density} proves that every $(3,6)$-tight graph
on at least five vertices is strongly $\Z_3$-connected, and hence improves
the critical-graph bound to $3n-7$ for $n\ge5$.

Approximation produces an integral boundary with the parity of the
vertex degrees. Its real edge values may slightly exceed one in
magnitude, but parity forces any violated cut inequality to fail by at
least two. We will bound the excess by less than two on every
\emph{bond}, meaning a cut whose two shores induce connected subgraphs.
The proof proceeds by lifting orientations across contractions, applying
the exact-boundary cut criterion, and then using this parity argument.

\subsection{Contractions and exact boundaries}

For later use, we allow a prescribed boundary. A pair $(J,\gamma)$
consists of a connected bridgeless graph and a balanced function
$\gamma:V(J)\to\Z_3$. A \emph{$\gamma$-orientation} is an orientation
$D$ with $d_D^+(v)-d_D^-(v)\equiv\gamma(v)\pmod3$ at every vertex.
The pair is \emph{feasible} if such an orientation exists and
\emph{infeasible} otherwise. It is \emph{contraction-critical} if it is
infeasible but every single-edge contraction is feasible. Boundary values
at identified vertices are summed; write $\gamma/e$ for the resulting
boundary on $J/e$. Thus $G$ is three-flow-critical exactly when $(G,0)$
is contraction-critical.

\begin{lemma}[Lifting a prescribed boundary]\label{lem:critical-reduced}
Let $J$ be connected, let $\gamma$ be a balanced boundary modulo three,
and let $Q$ be a strongly $\Z_3$-connected subgraph on at least two
vertices. Write $J/Q$ for the graph obtained by identifying $V(Q)$ and
deleting loops, with induced boundary $\gamma/Q$. Every
$(\gamma/Q)$-orientation of $J/Q$ lifts to a $\gamma$-orientation of $J$.
In particular, a $(\gamma/e)$-orientation of $J/e$ for any $e\in E(Q)$
also lifts to a $\gamma$-orientation of $J$. Consequently every contraction-critical pair
has a $\Z_3$-reduced graph on at least three vertices. In particular,
every three-flow-critical graph is $\Z_3$-reduced.
\end{lemma}
\begin{proof}
Take a $(\gamma/Q)$-orientation of $J/Q$. Its boundary at the identified
vertex is $\sum_{v\in V(Q)}\gamma(v)$. Lift the orientations of edges outside
$V(Q)$ and across its cut. Orient the edges with both endpoints in
$V(Q)$ but outside $Q$ arbitrarily. On $V(Q)$, subtract the contributions
of all these oriented edges from $\gamma$. The residual boundary has
sum zero, because internal edges contribute zero to the sum and the
quotient equation accounts for the cut edges. Strong
$\Z_3$-connectivity of $Q$ realizes the residual boundary and completes
the orientation of $J$.

An orientation of $J/e$ also gives one of $J/Q$ by further identification
and summation of the boundary equations. A contraction-critical pair
therefore cannot contain such a $Q$.
It cannot have one vertex, where the only balanced boundary is zero,
or two vertices: a bridgeless two-vertex graph has at least two parallel
edges, which can realize any residual boundary modulo three.
\end{proof}

For a vertex vector $\eta$, write $\eta(S)=\sum_{v\in S}\eta_v$. The next
lemma is the boundary form of Hakimi's orientation theorem~\cite{Hakimi}.
We include a proof, including the reduction from all cuts to bonds.

\begin{lemma}[Boundary cut criterion]\label{lem:boundary-cut-test}
Let $H$ be connected and let $\eta\in\Z^{V(H)}$ satisfy
$\eta(V(H))=0$ and $\eta_v\equiv\deg_H(v)\pmod2$. There is an
orientation with boundary exactly $\eta$ if and only if
\[
 |\eta(S)|\le|\delta_H(S)|
 \qquad\text{for every bond }\delta_H(S).
\]
\end{lemma}
\begin{proof}
Necessity follows by summing the boundary equations over a shore.
For sufficiency, the inequalities first extend to every cut. If a
shore violates one, some connected component of that shore violates
the same inequality. If its complement is disconnected, some component
of the complement violates the opposite inequality; its complement is
connected because $H$ is connected and the original shore was connected.
This gives a violated bond.

Set $b(v)=(\deg_H(v)+\eta_v)/2$. These are integers, and
\[
 b(V(H))=|E(H)|,\qquad
 b(S)=e_H(S)+\frac{|\delta_H(S)|+\eta(S)}2\ge e_H(S),
\]
where $b(S)=\sum_{v\in S}b(v)$. We will realize $b$ as the outdegree
vector. Choose an orientation $D$ minimizing
$\sum_v\max\{d_D^+(v)-b(v),0\}$.

If a vertex $s$ has $d_D^+(s)>b(s)$, let $R$ be the set of vertices
reachable from $s$ along directed paths. Some $t\in R$ must satisfy
$d_D^+(t)<b(t)$. Otherwise, since no edge leaves $R$,
\[
 e_H(R)=\sum_{v\in R}d_D^+(v)>b(R),
\]
contradicting the displayed inequality. Reversing a directed path from
$s$ to $t$ decreases the excess at $s$ by one, increases the outdegree
at $t$ by one, and preserves every internal outdegree. The sum of
positive excesses strictly decreases, a contradiction. Hence
$d_D^+(v)\le b(v)$ at every vertex. Their totals are equal, so equality
holds throughout, and the boundary is $2b-\deg_H=\eta$.
\end{proof}

\subsection{Uniform density and the parity argument}

\begin{lemma}[Uniform density and parity]\label{lem:dense-slack}
Let $H$ be a connected multigraph on $h\ge3$ vertices with $M\ge3h-5$
edges. Suppose that
\begin{equation}\label{eq:uniform-density}
 e_H(Y)\le\frac{M}{h-1}(|Y|-1)
 \qquad(\varnothing\ne Y\subsetneq V(H)).
\end{equation}
Then $H$ is strongly $\Z_3$-connected.
\end{lemma}
\begin{proof}
For a partition into $r\ge2$ parts, summing~\eqref{eq:uniform-density}
bounds the number of internal edges by $M(h-r)/(h-1)$. Thus there
are at least $M(r-1)/(h-1)$ crossing edges. The singleton partition
attains this ratio, so~\eqref{eq:packing} gives
$\tau_f(H)=M/(h-1)>1$.

Fix a balanced boundary $\beta$ modulo three. Let $B$ and $U$ be the
incidence and cycle-basis matrices from Section~\ref{sec:packing}.
Integer routing on a spanning tree gives $a\in\Z^{E(H)}$ with
$2Ba-B\one\equiv\beta\pmod3$. Apply the cycle-space approximation
\eqref{eq:cycle-approx} to $b=(\one/2-a)/3$, and set
\[
 u=6(Ux-b-z),\qquad \eta=Bu=2Ba-B\one-6Bz.
\]
The vector $\eta$ is integral, has total sum zero, and satisfies
\[
 \eta\equiv\beta\pmod3,\qquad
 \eta_v\equiv\deg_H(v)\pmod2,\qquad
 |u_e|\le\frac3{\tau_f(H)}
       =\frac{3(h-1)}M\le1+\frac2M.
\]
Let $C=\delta_H(S)$ be a bond of size $c$. Its two connected shores
contain at least $h-2$ edges, so $c\le M-h+2<M$. Consequently
\[
 |\eta(S)|\le\left(1+\frac2M\right)c<c+2.
\]
But $\eta(S)\equiv c\pmod2$, so $|\eta(S)|\le c$.
Lemma~\ref{lem:boundary-cut-test} gives an orientation with boundary
$\eta$, which realizes $\beta$.
\end{proof}

\begin{proof}[Proof of Theorem~\ref{thm:reduced-density}]
Two parallel edges are strongly $\Z_3$-connected, so a reduced graph is
simple. Suppose that some vertex set of size $h\ge3$ spans more than
$3h-6$ edges, and choose such a set minimal under inclusion.
Keep exactly $M=3h-5$ of its edges to form a subgraph $H$. Simplicity
gives $h\ge5$. Minimality implies $e_H(Y)\le3|Y|-6$ for every proper
subset with $3\le|Y|<h$. If $s=|Y|$, then
\[
 \frac{M}{h-1}(s-1)-(3s-6)
 =\frac{3h-2s-1}{h-1}>0.
\]
For $s=2$, simplicity gives $e_H(Y)\le1<M/(h-1)$; for $s=1$ both
sides of~\eqref{eq:uniform-density} vanish. Thus $H$ satisfies that
inequality for every nonempty proper subset. Summing it over the
components of $H$ shows that $H$ is connected, and
Lemma~\ref{lem:dense-slack} makes it strongly $\Z_3$-connected. This
contradicts reducedness and proves $(3,6)$-sparsity.

By Lemma~\ref{lem:critical-reduced}, every three-flow-critical graph
is reduced and has at least three vertices. Taking $Y=V(G)$ gives
$|E(G)|\le3|V(G)|-6$.
\end{proof}

The distinction between this proof and Appendix~\ref{app:critical-density}
is exact. When $M=3h-5$, the excess allowed by approximation is at most
$2c/M<2$ on each bond. At the tight count $M=3h-6$, it becomes $3c/M$.
A bond can then have $c=2h-4=2M/3$ edges, giving excess exactly two;
this happens precisely when its two shores are trees. The appendix
analyzes this equality case and removes it by perturbing the edgewise
approximation bounds.

\subsection{Small orders and the remaining conjecture}

The small cases explain both order restrictions in the density bounds.
They follow from the bound just proved and elementary orientations.

\begin{proposition}[Small orders and sharpness]\label{prop:critical-small-orders}
The only three-flow-critical graph on at most five vertices is $K_4$.
Two copies of $K_4$ sharing exactly one edge form a three-flow-critical
graph with six vertices and eleven edges. Thus $K_4$ attains $3n-6$,
and the six-vertex graph attains $3n-7$.
\end{proposition}
\begin{proof}
A three-flow-critical graph has at least three vertices by
Lemma~\ref{lem:critical-reduced} and is simple by
Theorem~\ref{thm:reduced-density}. Its minimum degree is at least three:
connectedness excludes degree zero, and bridgelessness excludes
degree one. If a vertex $v$ had degree two, with neighbours $x,y$, a
zero-boundary orientation of $G/vx$ would lift by replacing the edge
corresponding to $vy$ by a directed path through $v$. This preserves all
old boundary values and gives boundary zero at $v$, a contradiction.
Hence there is no critical graph on three vertices, and on four vertices
the only possibility is $K_4$.

The graph $K_4$ has no zero-boundary orientation modulo three: every
cubic vertex would have to be a source or a sink, forcing bipartiteness.
For any edge $e$, the graph $K_4/e$ is a triangle with two doubled sides.
Contracting a doubled side leaves three parallel edges. Both two and
three parallel edges realize every balanced boundary modulo three, so
Lemma~\ref{lem:critical-reduced} shows that $K_4/e$ is strongly
$\Z_3$-connected. Thus $K_4$ is critical.

On five vertices, minimum degree three and simplicity imply that the
complement is a matching. Theorem~\ref{thm:reduced-density} gives $m\le9$,
so this matching has one or two edges. If it has one edge $xy$, the graph
is $K_5-xy$. Direct every edge at $x$ away from $x$, every edge at $y$
toward $y$, and the triangle on the remaining three vertices cyclically.
This gives boundary zero modulo three. If the complement has two edges,
the graph is the wheel with four rim vertices. Orient the rim vertices
alternately as sources and sinks; the hub then has two edges in each
direction. Again the boundary is zero, excluding criticality.

For the six-vertex example, take adjacent vertices $a,b$ and two disjoint
pairs $x_i,y_i$, for $i=1,2$. Add each edge $x_i y_i$ and join both its
ends to both $a,b$; call the resulting graph $J$. In a zero-boundary
orientation the cubic vertices $x_i,y_i$ would be opposite sources and
sinks. Their contributions at each of $a,b$ would cancel, leaving only
the nonzero contribution of $ab$. Thus $J$ has no three-flow.
For any edge $e$, choose a copy of $K_4$ in $J$ that contains $e$.
Its image in $J/e$ is strongly $\Z_3$-connected, as just proved;
contracting that image leaves the other copy with $ab$ contracted,
which is also strongly $\Z_3$-connected. Lemma~\ref{lem:critical-reduced}
therefore shows that every $J/e$ is strongly $\Z_3$-connected.
Hence $J$ is critical and has $11=3\cdot6-7$ edges.
\end{proof}

The six-vertex example explains why the following conjecture begins at
seven vertices.
Li, Ma, Shi, Wang and Wu~\cite[Conjecture~1.3]{LMSWW} conjecture
$m\le3n-8$ for every three-flow-critical graph on $n\ge7$ vertices.
This would be sharp for every such $n$: add an edge joining two vertices
in the three-vertex part of the complete bipartite graph $K_{3,n-3}$. The resulting graph
$K^+_{3,n-3}$ is three-flow-critical~\cite[Proposition~2.6]{LMSWW} and
has $3(n-3)+1=3n-8$ edges. The appendix bound is therefore one edge
above the conjectured optimum; this conjecture remains open.

\section{Seymour's six-flow theorem}\label{sec:seymour}

\begin{theorem}[Seymour~\cite{Seymour}]\label{thm:seymour}
Every bridgeless graph has a nowhere-zero integer $6$-flow.
\end{theorem}
\begin{proof}
By the standard reduction for nowhere-zero flows~\cite[Section~2]{DFRS}, it suffices to consider connected cubic graphs (all degrees three) that are three-edge-connected. For such a graph, \eqref{eq:lambda-tau} gives
\[
 \tau_f(G)\ge\frac{3n}{2(n-1)}>\frac32.
\]
Theorem~\ref{thm:nonzero-connectivity} with $q=6$ supplies a nowhere-zero integer $6$-flow.
\end{proof}

\section{Circular flows}\label{sec:circular}
A \emph{circular $r$-flow}, for real $r\ge2$, is a real circulation with $1\le|f_e|\le r-1$. The infimum of such $r$ is the \emph{circular flow number} $F_c(G)$.

\begin{theorem}[Circular flows]\label{thm:circular}
Every connected bridgeless graph satisfies
\[
 F_c(G)\le2+\frac{2}{\tau_f(G)-1}.
\]
In particular, for every integer $s\ge1$, every $(4s+2)$-edge-connected graph has
\[
 F_c(G)<2+\frac1s
\]
and realizes every balanced boundary modulo $2s+1$ by a directed strongly connected orientation.
\end{theorem}
\begin{proof}
Bridgelessness gives $\lambda(G)\ge2$, so~\eqref{eq:lambda-tau} implies $\tau=\tau_f(G)>1$. Put $\alpha=1-1/\tau$. The construction of Theorem~\ref{thm:fractional} at $q=1$ gives $X\in[\alpha/2,1-\alpha/2]^E$ with integral boundary. Round it to $y\in\{0,1\}^E$ with $By=BX$. Then $(2/\alpha)(X-y)$ is a circulation with magnitudes between $1$ and $2/\alpha-1$, proving the first bound.

If $\lambda(G)\ge4s+2$, inequality~\eqref{eq:lambda-tau} gives $\tau\ge(2s+1)n/(n-1)>2s+1$. Hence
\[
 F_c(G)\le2+\frac{2(n-1)}{2sn+1}<2+\frac1s.
\]
The directed-strong conclusion follows from Theorem~\ref{thm:fractional} with $q=2s+1$ and the conversion from boundaries to outdegrees.
\end{proof}

At $s=1$, the balanced-boundary assertion recovers the six-edge-connected modulo-three orientation guarantee of Lov\'asz, Thomassen, Wu and Zhang~\cite{LTWZ}. Taking zero boundary gives their nowhere-zero $3$-flow consequence. The strict inequality $F_c(G)<3$ proves the conjecture of Li, Thomassen, Wu and Zhang~\cite[Conjecture~1.3]{LTWZindex}. Their Theorem~1.1 identifies $F_c(G)<2+1/s$ with the existence of a directed strongly connected realization of the zero boundary modulo $2s+1$. The assertion above allows an arbitrary balanced boundary.

\section{Antisymmetric flows}\label{sec:antisymmetric}
In a fixed orientation, a \emph{$\Z_q$-flow} $f$ is a circulation with values in $\Z_q$. Write $f(E)=\{f_e:e\in E\}$ and $-f(E)=\{-f_e:e\in E\}$. The flow is \emph{antisymmetric} if $f(E)\cap(-f(E))=\varnothing$; in particular, no value is zero or self-inverse.

\begin{theorem}[Antisymmetric flows]\label{thm:antisymmetric}
Let $s\ge2$ be an integer. If $\tau_f(G)\ge(2s+1)/(s-1)$, then, in every fixed orientation, every balanced boundary modulo $2s+1$ has a realization with edge values in $\{1,\ldots,s\}$. In particular, every orientation admits an antisymmetric $\Z_{2s+1}$-flow. The same conclusions hold if the edge-connectivity is at least $\lceil(4s+2)/(s-1)\rceil$. For $s=2$ this gives an antisymmetric $\Z_5$-flow in every orientation of a ten-edge-connected graph.
\end{theorem}
\begin{proof}
Apply Theorem~\ref{thm:interval} with $q=2s+1$ and intervals $[1,s]$, then take zero boundary. The edge-connectivity hypothesis implies the fractional one by~\eqref{eq:lambda-tau}.
\end{proof}
Esperet, de Joannis de Verclos, Le and Thomass\'e~\cite{EJLT} gave the sufficient edge-connectivity $\lceil6s/(s-1)\rceil$, which is twelve at modulus five. The ten-edge-connected conclusion also removes the planarity assumption from the corresponding result of Cranston, Li, Su, Wang and Wei~\cite{CLSWWei}.

\section{Signed modular orientations}\label{sec:signed}

A \emph{signed graph} assigns a sign to each edge; this assignment is its \emph{signature}. Here loops and one-vertex graphs are allowed. Each edge has two \emph{half-edges}, one at each endpoint, or both at the same vertex for a loop.
In a signed orientation, a positive edge has one outgoing and one incoming
half-edge, while a negative edge has either both outgoing or both incoming.
The signed boundary is $\partial_D(v)=d_D^+(v)-d_D^-(v)$, counting half-edges;
thus a negative loop contributes $2$ or $-2$. A \emph{modulo $k$-orientation}
has $\partial_D(v)\equiv0\pmod k$ at every vertex.

Switching at a vertex reverses its incident half-edges and the signs of its
incident non-loop edges. It negates the boundary at that vertex and leaves the other boundary coordinates unchanged; in particular, zero boundary is preserved. Let $\epsilon(G)$, the \emph{negativeness}, be the minimum number of negative
edges among all switching-equivalent signatures. The graph is
\emph{essentially $k$-unbalanced} if every such signature has either an even
number of negative edges or at least $k$ negative edges. Edge-connectivity
refers to the underlying unsigned graph.

For a signed graph without positive loops, fix a reference orientation
with every negative edge directed outward. Its incidence matrix $B$ has
columns $e_u-e_v$ for positive edges $u\to v$, $e_u+e_v$ for negative
edges $uv$, and $2e_v$ for negative loops at $v$; the $e_v$ are coordinate
vectors. For $y\in\{0,1\}^E$, keep the reference direction on edges with $y_e=1$ and reverse both half-edges when $y_e=0$; the resulting boundary is $B(2y-\one)$. Write $B\Z^E=\{Bz:z\in\Z^E\}$ for the integer boundary lattice.

Subdividing positive edges converts signed boundary equations into degree equations in an ordinary graph. A uniform margin from $0$ and $1$ will give the odd-set inequalities used to round the resulting fractional edge values.

\begin{lemma}[Signed interior rounding]\label{lem:signed-rounding}
Let $G$ have at least two vertices and no positive loops, and take $B$ as
above. If $0<\alpha\le\tfrac12$, $\alpha\lambda(G)\ge1$, and
\begin{equation}\label{eq:signed-rounding-hyp}
 X\in[\alpha,1-\alpha]^E,\qquad BX\in B\Z^E,
\end{equation}
then there exists $y\in\{0,1\}^E$ with $By=BX$.
\end{lemma}

\begin{proof}
Put $h=BX$. For a vertex function $f$, write $f(W)=\sum_{v\in W}f(v)$.
Subdivide every positive edge $u\to v$ into $uw_e,w_ev$,
with values $X_e,1-X_e$, and retain each negative non-loop edge with
value $X_e$. Replace a negative loop $e$ at $v$ by a triangle on $v$
and two new vertices. Give the two edges at $v$ value $X_e$, the third
edge value $1-X_e$, and both new vertices demand one.
This gives a loopless multigraph $J$ with fractional edge vector $x$ and
integer degree demands
\[
 f(v)=h_v+d^-_{0,+}(v),\qquad f(w_e)=1,
\]
where $v$ ranges over the original vertices, $w_e$ ranges over the subdivision vertices, and $d^-_{0,+}$ counts incoming reference positive edges. The demands are
nonnegative because $x$ realizes them. Each positive-edge subdivision
and each loop triangle contributes two to the total demand in addition
to the terms $h_v$. Thus $f(V(J))$ is even, since $h\in B\Z^E$ and every
column of $B$ has even coordinate sum.

An $f$-factor is a spanning subgraph with degree $f(v)$ at each vertex,
and the convex hull of its binary edge vectors is
described~\cite[Appendix~C, p.~56 of the arXiv version]{Gabow} by the degree equations, $0\le x\le1$, and, for every $W\subseteq V(J)$,
\begin{equation}\label{eq:signed-blossom}
 \sum_{e\in\delta_J(W)\setminus F}x_e+
 \sum_{e\in F}(1-x_e)\ge1
 \quad\text{if }F\subseteq\delta_J(W),\ f(W)+|F|\text{ is odd}.
\end{equation}
If $W\cap V(G)$ is a nonempty proper shore, each edge of $\delta_G(W\cap V(G))$ contributes a crossing edge of $J$: a retained negative edge, or at least one of the two edges replacing a positive edge. Thus the left side is at least
$\alpha|\delta_J(W)|\ge\alpha\lambda(G)\ge1$.
Suppose $W$ contains no original vertex. Separate $W$ into the
subdivision vertices it contains and the nonempty intersections with the
pairs of new vertices in the loop triangles. Each piece has its own
crossing edges, and these sets partition $\delta_J(W)$. A piece of size
one has demand one and two crossing edges with complementary values.
It contributes one to the left side if the number of those edges in $F$
is even. A piece of size two has demand two and two crossing edges with
equal values. It contributes one if exactly one of those edges is in $F$.
Thus any piece whose demand plus its number of edges in $F$ is odd
contributes one. The parity condition in~\eqref{eq:signed-blossom}
guarantees such a piece. If $W$ contains every original vertex, apply
this argument to its complement, using the even total demand.
Thus all inequalities hold, so an integral $f$-factor exists. For each positive edge $e=u\to v$, take $y_e$ to be the factor indicator of $uw_e$. The indicator of $w_ev$ is then $1-y_e$ because $f(w_e)=1$. For each retained negative edge, copy its factor indicator to $y_e$. In each loop triangle the two new
vertices have degree one in the factor, so either both edges at the
original vertex are selected or only the edge between the new vertices
is selected. Set $y_e$ to one in the first case and zero in the second.
The original-vertex degree equations give $By=h$.
\end{proof}

\begin{theorem}[Signed boundary prescriptions]\label{thm:signed-boundaries}
Let $k\ge3$ be odd. If a signed graph $G$ on at least two vertices satisfies
\[
 \lambda(G)\ge2k+2,\qquad \epsilon(G)\ge k-1,
\]
then every function $\beta:V(G)\to\Z/k\Z$ is realized by a signed
orientation $D$ with $\partial_D(v)\equiv\beta(v)\pmod k$ for every $v$.
There is no zero-sum condition on $\beta$.
\end{theorem}

\begin{proof}
Delete positive loops, to be restored arbitrarily. If $t=\epsilon(G)\in\{k-1,k\}$,
switch to a signature with exactly $t$ negative edges and negate $\beta$
at the switched vertices. For any partition into $r\ge2$ parts, the
positive spanning subgraph has at least $r\lambda(G)/2-t\ge k(r-1)$
crossing edges. It therefore contains $k$ edge-disjoint spanning trees.
Orienting $j$ negative edges outward and $t-j$ inward contributes total
boundary $4j-2t$. Every $j\in\{0,\ldots,k-1\}$ is feasible because $t\ge k-1$;
these totals cover every residue modulo $k$, since $k$ is odd. Choose the total to match the switched
prescription and use Theorem~\ref{thm:ordinary-residues} on the positive
subgraph to realize the balanced residual boundary. Switching back
gives the required orientation.

Now assume $\epsilon(G)\ge k+1$ and take $B$ as above.
A signed component is \emph{balanced} if it can be
switched to all positive edges; let $b(H)$ count the balanced components
of $H$, including isolated vertices. The signed incidence rank is
$|V(H)|-b(H)$: a left-kernel vector has equal values across positive
edges and opposite values across negative edges, with one free value on
each balanced component and zero on every unbalanced component.
Since $\epsilon(G)>0$, $G$ is unbalanced. Choose independent integer
columns $U$ spanning $\ker B$, by clearing denominators in a rational
nullspace basis. Restriction to $F\subseteq E$ gives
\[
 \rank(U_F)=|F|-b(G-F).
\]
Indeed, $G$ is connected and unbalanced, so $\dim\ker B=|E|-|V|$. The kernel of restriction to $F$ consists of vectors supported on $G-F$ with zero signed boundary, a space of dimension $|E|-|F|-|V|+b(G-F)$. Subtracting gives the rank formula.
If every balanced component of $G-F$ is a proper shore, summing its cut
sizes gives $b(G-F)\lambda(G)\le2|F|$. The only other case is a balanced
spanning component: switching it positive puts every negative edge of $G$
in $F$, so $|F|\ge\epsilon(G)$. In either case,
\[
 \rank(U_F)\ge\frac{k}{k+1}|F|.
\]
For nonempty $F$ this bound also ensures that $U$ has no zero rows and a
nonzero column space. Lemma~\ref{lem:approx} therefore gives, for any
shift, an error at most $1/[2(k+1)]$.

The map $(\Z_k)^E\to(\Z_k)^V$ induced by $B$ is onto. Switching and changing reference directions act by row and column sign changes, so we may switch a spanning tree positive and take every negative edge outward. The tree columns, with any root row removed, give a square matrix of determinant $\pm1$ by successive elimination of leaves. Adjoin a negative edge column to the full tree columns and replace the root row by the sum of all rows. The new root row is zero in the tree columns and has entry $2$ in the negative edge column, including for a negative loop. Expansion along that row gives determinant $\pm2$, which is invertible modulo odd $k$.
Choose $u\in\Z^E$ with $Bu\equiv(\beta+B\one)/2\pmod k$.
Apply Lemma~\ref{lem:approx} with shift $(\tfrac12\one-u)/k$, obtaining
$x,z$, and put
\[
 X=\tfrac12\one+k\left(Ux-\frac{\tfrac12\one-u}{k}-z\right)
   =u+kUx-kz.
\]
Then $BX=Bu-kBz\in B\Z^E$ and
$X\in[\alpha,1-\alpha]^E$ for $\alpha=1/[2(k+1)]$.
Since $\alpha\lambda(G)\ge1$, Lemma~\ref{lem:signed-rounding} gives $y$ with
$By=BX$, and hence $B(2y-\one)\equiv\beta\pmod k$.
\end{proof}

\begin{theorem}[Essentially unbalanced signed graphs]\label{thm:signed-modulo}
Let $k\ge3$ be odd. Every essentially $k$-unbalanced signed graph in which
every cut separating two nonempty vertex sets has size at least $2k+2$ has a modulo $k$-orientation.
The cut hypothesis is vacuous for a one-vertex graph.
\end{theorem}

\begin{proof}
For at least two vertices, delete positive loops and choose a minimum
signature, with $t=\epsilon(G)$ negative edges. If $t\ge k-1$, apply
Theorem~\ref{thm:signed-boundaries} with $\beta=0$. Otherwise $t\le k-2$
and essential unbalance makes $t$ even.

Let $H$ be the spanning subgraph of positive edges. For any partition
$\mathcal P$ into $r\ge2$ parts, summing the cut sizes and deleting at
most $t$ crossing edges gives
\[
 e_H(\mathcal P)\ge\frac{r\lambda(G)}2-t
                  \ge kr-t\ge k(r-1).
\]
The Nash--Williams--Tutte theorem~\cite{NWtrees,Tutte} therefore supplies
$k$ edge-disjoint spanning trees in $H$. Orient half the
negative edges with both half-edges outgoing and half with both incoming.
Their signed boundary $b$ has total zero and is therefore balanced modulo $k$.
Theorem~\ref{thm:ordinary-residues} orients $H$ with boundary $-b$
modulo $k$. Combining these orientations and switching back gives zero
boundary; restore positive loops arbitrarily.

For one vertex, if the number
$r$ of negative loops is even, orient half outward; if $r$ is odd,
essential unbalance gives $r\ge k$, and orient $(r+k)/2$ outward.
The resulting boundary is respectively $0$ or $2k$.
\end{proof}

\paragraph{Comparison.}
For arbitrary boundary prescriptions, Zhu~\cite[Theorem~3]{Zhu} assumes
$\epsilon(G)\ge k$ and edge-connectivity $6k-7$; Theorem~\ref{thm:signed-boundaries}
uses $\epsilon(G)\ge k-1$ and edge-connectivity $2k+2$.
For modulo $k$-orientations under essential
$k$-unbalance, Zhu's bound is also $6k-7$ \cite[Theorem~4]{Zhu}.
Lu, Miao, Tian and Zhang~\cite[Theorem~1.4]{LMTZ} improve it to $6k-10$. Our bound
$2k+2$ agrees with $6k-10$ when $k=3$ and improves it for every odd
$k\ge5$.

\clearpage
\appendix
\section{Self-contained proof of Lemma~\ref{lem:approx}}\label{app:approx}
\setcounter{equation}{0}
\renewcommand{\theequation}{A.\arabic{equation}}

We prove the shifted approximation bound directly by choosing positive weights and then minimizing a weighted squared error. The log-determinant weights follow Barthe~\cite{Barthe} and Alon, Buci\'c and Davies~\cite{ABD}. We use the exchange argument of the latter, with a slack parameter to handle tight rank inequalities by a limit. The proof below uses neither their lemma nor matroid partition.

Let $A\in\Z^{m\times d}$ have independent columns, $d\ge1$, and no zero rows, as in Lemma~\ref{lem:approx}. Write $a_i^{\mathsf T}$ for row $i$ of $A$. Recall that
\[
 \mu=\min_{\varnothing\ne F\subseteq[m]}\frac{\rank(A_F)}{|F|}\in(0,1].
\]
Fix $0<\rho<\mu$. We first obtain error at most $(1-\rho)/2$ and then pass to the endpoint $\rho=\mu$.

\subsection{Choosing the weights}

For $t=(t_1,\ldots,t_m)\in[0,\infty)^m$, define
\begin{equation}\label{eq:app-objective}
 M(t)=\sum_{i=1}^m e^{t_i}a_i a_i^{\mathsf T},
 \qquad F_\rho(t)=\log\det M(t)-\rho\sum_{i=1}^m t_i.
\end{equation}
The rows span $\R^d$, so $M(t)$ is positive definite: for $v\ne0$, the sum
$v^{\mathsf T}M(t)v=\sum_i e^{t_i}(a_i^{\mathsf T}v)^2$ is positive.

We show that $F_\rho$ attains a minimum. Let $\mathcal B$ be the sets of $d$ independent rows. Cauchy--Binet gives
\begin{equation}\label{eq:app-CB}
 \det M(t)=\sum_{J\in\mathcal B}(\det A_J)^2
                       \exp\!\left(\sum_{i\in J}t_i\right).
\end{equation}
Each $\det A_J$ is a nonzero integer, so its square is at least one.

For the current $t$, reorder the indices so that $t_1\ge\cdots\ge t_m\ge0$, and set $t_{m+1}=0$. Scan the rows in this order, accepting a row precisely when it increases their span. The accepted rows form a basis $J$. Among the first $j$ rows, exactly $r_j=\rank(A_{[j]})$ have been accepted, since the accepted rows span all rows already scanned. Therefore
\[
 \sum_{i\in J}t_i
 =\sum_{j=1}^m(t_j-t_{j+1})r_j
 \ge\mu\sum_{j=1}^m(t_j-t_{j+1})j
 =\mu\sum_{i=1}^m t_i.
\]
Keeping the term indexed by $J$ in~\eqref{eq:app-CB} now gives
\begin{equation}\label{eq:app-coercive}
 F_\rho(t)\ge(\mu-\rho)\sum_i t_i.
\end{equation}
Thus the sublevel set $F_\rho(t)\le F_\rho(0)$ is nonempty, closed and bounded. Minimizing on this compact set gives a global minimum $t^*$ on the nonnegative orthant.

Set $w_i=e^{t_i^*}$, let $W=\operatorname{diag}(w_1,\ldots,w_m)$ be the diagonal matrix with these entries, and put $M=A^{\mathsf T}WA$. Differentiating~\eqref{eq:app-objective} gives
\[
 \frac{\partial F_\rho}{\partial t_i}(t^*)=\ell_i-\rho,
 \qquad \ell_i=w_i a_i^{\mathsf T}M^{-1}a_i.
\]
Indeed, $\partial M/\partial t_i=w_i a_i a_i^{\mathsf T}$ and
$\partial\log\det M=\operatorname{tr}(M^{-1}\partial M)$, where $\operatorname{tr}$ denotes the sum of the diagonal entries.
If $t_i^*>0$, the derivative is zero; if $t_i^*=0$, its right derivative is nonnegative. Hence
\begin{equation}\label{eq:app-leverage}
 \ell_i\ge\rho\qquad(i\in[m]).
\end{equation}

\paragraph{Spanning-tree interpretation.}
For a fundamental cycle basis $A=U$ of a graph, the nonzero maximal row minors are indexed by complements of spanning trees and have absolute determinant one. Indeed, the coordinates outside any spanning tree give integral coordinates on the lattice of integer circulations, while a cycle supported on the complement of the selected rows makes the restriction singular. With $c_e=e^{-t_e}$, the Cauchy--Binet expansion therefore gives
\[
 F_\rho(t)=(1-\rho)\sum_e t_e+
           \log\sum_T\prod_{e\in T}c_e ,
\]
where $T$ ranges over spanning trees. Differentiating shows that $1-\ell_e$ is the probability that $e$ belongs to a tree drawn with probability proportional to $\prod_{f\in T}c_f$. Thus the orthant minimization bounds these probabilities by $1-\rho$. The endpoint argument below concerns approximation errors and does not require the weights to converge.

\subsection{Minimizing the approximation error}

Fix an arbitrary $b\in\R^m$ and consider
\[
 \Phi(x,z)=\sum_{i=1}^m w_i(a_i^{\mathsf T}x-b_i-z_i)^2,
 \qquad x\in\R^d,\quad z\in\Z^m.
\]
This function has a globally minimizing pair. For fixed $x$, choosing each $z_i$ to be a nearest integer gives
\[
 \min_{z\in\Z^m}\Phi(x,z)
 =\sum_i w_i\dist(a_i^{\mathsf T}x-b_i,\Z)^2.
\]
The right side is continuous and $\Z^d$-periodic because $A$ is integral. It attains its global minimum on $[0,1]^d$; choose nearest integers there to obtain a minimizing pair $(x,z)$.

Put $r=Ax-b-z$. With $z$ fixed, differentiating the smooth quadratic $\Phi$ at its minimum gives
\begin{equation}\label{eq:app-normal}
 A^{\mathsf T}Wr=0.
\end{equation}
Fix $i$ and let $e_i$ be the $i$th coordinate vector of $\R^m$. A unit change in $z_i$ is compensated by the weighted least-squares adjustment with normal equation $Mh_i=w_i a_i$. Set $h_i=w_iM^{-1}a_i$. For either sign $\varepsilon\in\{-1,1\}$, compare $(x,z)$ with $(x+\varepsilon h_i,z+\varepsilon e_i)$. Expanding the square and using~\eqref{eq:app-normal},
\begin{align*}
 &\Phi(x+\varepsilon h_i,z+\varepsilon e_i)-\Phi(x,z)\\
 &\qquad=-2\varepsilon w_i r_i+h_i^{\mathsf T}Mh_i
                         -2w_i a_i^{\mathsf T}h_i+w_i\\
 &\qquad=w_i(1-\ell_i-2\varepsilon r_i).
\end{align*}
Here $h_i^{\mathsf T}Mh_i=w_i a_i^{\mathsf T}h_i=w_i\ell_i$. Global minimality makes the difference nonnegative for both signs, so~\eqref{eq:app-leverage} yields
\begin{equation}\label{eq:app-residual}
 |r_i|\le\frac{1-\ell_i}{2}\le\frac{1-\rho}{2}
 \qquad(i\in[m]).
\end{equation}

\paragraph{Finite descent for fixed weights.}
In the graphic case $A=U$, suppose positive weights satisfying
\eqref{eq:app-leverage} are supplied. Set $\Psi(z)=\min_x\Phi(x,z)$.
The compensated exchange above attains the new least-squares minimum,
so its displayed difference is also
$\Psi(z+\varepsilon e_i)-\Psi(z)$.
Moreover, $\Psi(z)$ depends only on $Bz$, since $\im U=\ker B$.
If $\Psi(z)\le C$, its residual $r$ satisfies $r^{\mathsf T}Wr\le C$
and $Bz=-Bb-Br$; hence $Bz$ ranges over a bounded set of integer vectors.
Thus strictly improving unit shifts terminate, and at termination
\eqref{eq:app-residual} holds. This is a finite search with exact comparisons
for supplied weights and $\rho<\mu$; no polynomial bound on its length
is established. The closed endpoint still follows by the limit below.

\subsection{The endpoint}

The continuous, $\Z^d$-periodic function
\[
 H(x)=\max_{i\in[m]}\dist(a_i^{\mathsf T}x-b_i,\Z)
\]
attains a minimum on $[0,1]^d$. By~\eqref{eq:app-residual}, this minimum is at most $(1-\rho)/2$ for every $0<\rho<\mu$, and hence at most $(1-\mu)/2$. Choose a minimizing $x$ and nearest integers $z_i$. Then
\[
 \|Ax-b-z\|_\infty=H(x)\le\frac{1-\mu}{2},
\]
which proves Lemma~\ref{lem:approx}, including the closed endpoint.\hfill$\square$

\clearpage
\section{The bound \texorpdfstring{$3n-7$}{3n-7} for three-flow-critical graphs}\label{app:critical-density}
\setcounter{equation}{0}
\renewcommand{\theequation}{B.\arabic{equation}}

The definitions of $(3,6)$-sparsity and tightness are given in Section~\ref{sec:critical-density}. We now exclude equality in the reduced-density bound for graphs on at least five vertices.

\begin{theorem}\label{thm:tight-strong}
Every $(3,6)$-tight graph on at least five vertices is strongly $\Z_3$-connected. Consequently every $\Z_3$-reduced graph, and in particular every three-flow-critical graph, on $n\ge5$ vertices has at most $3n-7$ edges.
\end{theorem}

The restriction $n\ge5$ is necessary: $K_4$ is tight but has no zero-boundary orientation modulo three. Indeed, in such an orientation of a cubic graph every vertex would have to be a source or a sink, forcing the graph to be bipartite.

The proof has two parts. A perturbation of the approximation bound handles a degree-five vertex whose prescribed boundary is zero. Local reductions handle every other possible vertex of degree at most five. The five-vertex base case is proved directly; no enumeration is used.

\subsection{Approximation with separate edge bounds}

The equality case requires different error bounds on different edges.
The next lemma allows this redistribution whenever the bounds satisfy
the spanning-tree inequalities below.

Let $H$ be a connected bridgeless graph on $h\ge2$ vertices, with $m$ edges, incidence matrix $B$, and integer cycle-basis matrix $U$ as in Section~\ref{sec:packing}. Put $d=m-h+1$, the number of columns of $U$. For any edge vector $w$ and edge set $F$, write $w(F)=\sum_{e\in F}w_e$. Write $E_H(X)$ for the set of edges with both endpoints in $X$.

\begin{lemma}[Coordinatewise approximation]\label{lem:coordinate-cover}
Suppose an edge vector $p$ satisfies
\begin{equation}\label{eq:edge-bound-hypotheses}
 \begin{gathered}
 0<p_e<1\quad(e\in E(H)),\qquad p(E(H))=h-1,\\
 p(E_H(X))\le |X|-1\quad(\varnothing\ne X\subseteq V(H)).
 \end{gathered}
\end{equation}
For every $b\in\R^{E(H)}$, there exist $x\in\R^d$ and $z\in\Z^{E(H)}$ such that
\begin{equation}\label{eq:coordinate-approx}
 |(Ux-b-z)_e|\le p_e/2\qquad(e\in E(H)).
\end{equation}
Consequently, for every balanced boundary $\beta:V(H)\to\Z_3$, there are $u\in\R^{E(H)}$ and $\eta\in\Z^{V(H)}$ with
\begin{equation}\label{eq:coordinate-boundary}
 |u_e|\le3p_e,\qquad \eta=Bu,\qquad
 \eta\equiv\beta\pmod3,\qquad
 \eta(v)\equiv\deg_H(v)\pmod2.
\end{equation}
\end{lemma}
\begin{proof}
\emph{Rank inequalities.} Put $c_e=1-p_e$. These complementary bounds
replace the uniform rank ratio in Appendix~\ref{app:approx}. If $H-F$
has components with vertex sets $X_1,\ldots,X_r$, then
\[
 p(E(H)\setminus F)\le\sum_{j=1}^r p(E_H(X_j))
 \le\sum_{j=1}^r(|X_j|-1)=h-r.
\]
Thus $p(F)\ge r-1$, and the cycle-space rank identity gives
\begin{equation}\label{eq:coordinate-rank}
 c(F)=|F|-p(F)\le |F|-r+1=\rank(U_F).
\end{equation}

\emph{Choosing weights.} Write $\xi_e^{\mathsf T}$ for row $e$ of $U$.
Fix $0<\varepsilon<1$, and for $t\in[0,\infty)^m$ define
\[
 M(t)=\sum_e e^{t_e}\xi_e\xi_e^{\mathsf T},\qquad
 F_\varepsilon(t)=\log\det M(t)-(1-\varepsilon)\sum_e c_et_e.
\]
The matrix $M(t)$ is positive definite because the columns of $U$ are independent. To see that $F_\varepsilon$ has a minimum, order the edges so that $t_1\ge\cdots\ge t_m\ge0$, set $t_{m+1}=0$, and write $[j]=\{1,\ldots,j\}$. Greedily select a basis $J$ of the rows in this order. Cauchy--Binet and integrality of $U$, followed by Abel summation and~\eqref{eq:coordinate-rank}, give
\begin{align*}
 \log\det M(t)
 &\ge\sum_{i\in J}t_i
   =\sum_{j=1}^m(t_j-t_{j+1})\rank(U_{[j]})\\
 &\ge\sum_{j=1}^m(t_j-t_{j+1})c([j])
   =\sum_{e}c_et_e.
\end{align*}
Since every $c_e>0$, the bound
$F_\varepsilon(t)\ge\varepsilon\sum_e c_et_e$ makes each sublevel set bounded. Continuity therefore gives a minimizing point $t^*$.

Set $w_e=e^{t_e^*}$, $W=\operatorname{diag}(w_e)$ and $M=U^{\mathsf T}WU$. As in~\eqref{eq:app-leverage}, the one-sided derivative conditions at $t^*$ give
\[
 \ell_e:=w_e\xi_e^{\mathsf T}M^{-1}\xi_e
 \ge(1-\varepsilon)c_e.
\]
\emph{Bounding the error.} For these weights, minimize
\[
 \Phi(x,z)=(Ux-b-z)^{\mathsf T}W(Ux-b-z),
 \qquad x\in\R^d,\quad z\in\Z^{E(H)}.
\]
A minimizer exists: after choosing nearest integers $z$, the objective is a continuous $\Z^d$-periodic function of $x$. At a minimizing pair put $r=Ux-b-z$, so $U^{\mathsf T}Wr=0$. Let $e_i$ be the $i$th coordinate vector of $\R^{E(H)}$ and $h_i=w_iM^{-1}\xi_i$. For either sign $s\in\{-1,1\}$, the exchange computation in Appendix~\ref{app:approx} reads
\[
 \Phi(x+s h_i,z+s e_i)-\Phi(x,z)
   =w_i(1-\ell_i-2s r_i)\ge0.
\]
Using both signs yields
\[
 |r_i|\le\frac{1-\ell_i}{2}
 \le\frac{1-(1-\varepsilon)c_i}{2}.
\]
Since $U$ is integral, reducing each resulting $x$ modulo $\Z^d$ into
$[0,1]^d$ does not change these distances to integers. Along a sequence
$\varepsilon\downarrow0$, compactness gives a convergent subsequence.
Continuity of distance to $\Z$ shows that its limit has distance at most
$(1-c_e)/2=p_e/2$ in each coordinate. Choosing nearest integers
proves~\eqref{eq:coordinate-approx}.

\emph{Prescribing the boundary.} Integer routing along a spanning tree
gives $a\in\Z^{E(H)}$ with $2Ba-B\one\equiv\beta\pmod3$:
multiplication by two is invertible modulo three, and the prescribed
vertex vector has sum zero. Apply~\eqref{eq:coordinate-approx} to
$b=(\one/2-a)/3$, and set
\[
 u=6(Ux-b-z),\qquad
 \eta=Bu=2Ba-B\one-6Bz.
\]
The error bound gives $|u_e|\le3p_e$, and the displayed formula gives integrality and the required congruence modulo three. Its parity is correct because $(B\one)_v\equiv\deg_H(v)\pmod2$.
\end{proof}

\subsection{The degree-five case}

For a tight graph on $h$ vertices the uniform edge bound is $a=(h-1)/(h-2)=1+1/(h-2)$. A bond can use the full two units permitted by parity only when both shores are trees; we first determine the boundary forced by this equality, and then perturb the edge bounds so that equality becomes impossible at a degree-five vertex of prescribed boundary zero.

We will repeatedly use two elementary consequences of tightness when $h\ge5$. Deleting a vertex and applying sparsity to the remaining vertices gives minimum degree at least three. A cut with shores $S,T$, both of size at least three, has at least
$(3h-6)-(3|S|-6)-(3|T|-6)=6$ edges; simplicity and minimum degree give at least three for a smaller shore. In particular, the graph is connected and bridgeless.

\begin{lemma}\label{lem:tight-degree-five}
Let $H$ be $(3,6)$-tight, let $v$ have degree five, and let $\beta$ be a balanced boundary with $\beta(v)=0$. Then $H$ has a $\beta$-orientation.
\end{lemma}
\begin{proof}
Write $h=|V(H)|$, $m=3h-6$, and $a=(h-1)/(h-2)$. Here $h\ge6$ because $H$ is simple and has a vertex of degree five.

\emph{Equality at a bond.} Suppose $u\in\R^{E(H)}$ satisfies $|u_e|\le a$, and $\eta=Bu$ is integral with $\eta(w)\equiv\deg_H(w)\pmod2$. If a bond $C=\delta_H(S)$ satisfies $\eta(S)\ge |C|+2$, then both shores are trees and, writing $T=V(H)\setminus S$,
\begin{equation}\label{eq:tree-bond-boundary}
 \eta(S)=2h-2,\qquad
 \eta(w)=
 \begin{cases}
 d_C(w)+2-d_S(w),&w\in S,\\
 -\bigl(d_C(w)+2-d_T(w)\bigr),&w\in T.
 \end{cases}
\end{equation}
Here $d_C(w)$ counts the crossing edges at $w$, and $d_S(w)$ or $d_T(w)$ is its degree inside its shore.

To prove this claim, connectedness of the two shores gives $|C|\le m-(h-2)=2h-4$. Therefore
\[
 |C|+2\le\eta(S)\le a|C|\le |C|+2.
\]
Equality forces $|C|=2h-4$, tree shores, and $\eta(S)=2h-2$. Moreover, each crossing value of $u$, read in the direction from $S$ to $T$, is $a$. Neither shore is a singleton, since $\deg_H(w)\le h-1<2h-4$.

Delete an edge of the tree $H[S]$ and let $F$ be one resulting component. If $b_F$ counts the edges of $C$ incident with $F$, minimum degree three gives
\[
 3|F|\le2(|F|-1)+b_F+1,
\]
so $b_F\ge2$. The same argument on $S\setminus F$ gives $b_F\le2h-6$.
Only one edge of $\delta_H(F)$ is outside $C$, and its value has magnitude
at most $a$. Therefore
\[
 b_F-1<a(b_F-1)\le\eta(F)\le a(b_F+1)<b_F+3.
\]
The strict inequalities use $a=1+1/(h-2)$ and $2\le b_F\le2h-6$.
Parity gives $\eta(F)\equiv b_F+1\pmod2$, so the only possible value
is $\eta(F)=b_F+1$. For a vertex $w\in S$, subtract these identities
for the components of $H[S]-w$ from $\eta(S)=2h-2$. Their crossing
counts sum to $|C|-d_C(w)$ and their number is $d_S(w)$, giving the first
vertex formula in~\eqref{eq:tree-bond-boundary}. Reversing the signs
gives the formula on $T$ and proves the claim.

\emph{An admissible perturbation.} The constant vector $p^0_e=a/3$
satisfies~\eqref{eq:edge-bound-hypotheses}. Its total is $h-1$, and for
$3\le|X|<h$,
\[
 p^0(E_H(X))\le\frac{(h-1)(|X|-2)}{h-2}<|X|-1.
\]
The inequality is also strict for $|X|=2$ by simplicity. Increase the
bounds on the five edges at $v$, compensating by a uniform decrease
on all $m$ edges. This defines the perturbation
\[
 \Delta_e=\begin{cases}1-5/m,&e\text{ is incident with }v,\\-5/m,&\text{otherwise.}\end{cases}
\]
Then $\Delta(E(H))=0$. All proper constraints involving at least two
vertices are strict at $p^0$, while singleton constraints are identically
zero. There are finitely many constraints, and $0<p^0_e<1$ for every
edge. Hence $p^t=p^0+t\Delta$ satisfies~\eqref{eq:edge-bound-hypotheses}
for every sufficiently small $t>0$.

\emph{Fixing a violated bond.} Suppose no $\beta$-orientation exists.
For a sequence $t\downarrow0$, use Lemma~\ref{lem:coordinate-cover}
with $p^t$ to obtain $u^t,\eta^t$ satisfying~\eqref{eq:coordinate-boundary}.
An orientation with exact boundary $\eta^t$ would realize $\beta$,
so Lemma~\ref{lem:boundary-cut-test} supplies a violated bond for each
$t$. Choose its shore with positive boundary; parity makes the violation
at least two.

The vectors $u^t$ are bounded, so the integral vectors $\eta^t=Bu^t$
take only finitely many values. There are also only finitely many shores.
Thus an infinite subsequence has the same boundary $\eta$ and the same
violating shore $S$. Take a further subsequence so that $u^t\to u$.
For $C=\delta_H(S)$, the limit satisfies
\[
 Bu=\eta,\qquad |u_e|\le a,\qquad \eta(S)\ge|C|+2.
\]
The equality claim therefore applies: $|C|=2h-4$, both shores are trees,
and~\eqref{eq:tree-bond-boundary} holds.

Let $d_{\rm sh}(v)$ be the degree of $v$ in its tree shore. Since $\eta(v)\equiv\beta(v)=0\pmod3$, formula~\eqref{eq:tree-bond-boundary} and $\deg_H(v)=5$ give
\[
 7-2d_{\rm sh}(v)\equiv0\pmod3,
 \qquad d_{\rm sh}(v)\equiv2\pmod3,
 \qquad d_C(v)\le3.
\]
Consequently
\[
 \Delta(C)=d_C(v)-\frac{5|C|}{m}
 \le3-\frac{10}{3}=-\frac13.
\]
The perturbation thus decreases the total allowance on this bond. Indeed, every positive $t$ in the chosen subsequence satisfies
\[
 |C|+2=\eta(S)\le3p^t(C)
 =|C|+2+3t\Delta(C)<|C|+2,
\]
which is impossible.
\end{proof}

\subsection{Contraction and local reductions}

For a graph $J$ define its potential and the saving from contracting a vertex set by
\[
 P(J)=3|V(J)|-|E(J)|,\qquad
 \sigma_J(X)=3(|X|-1)-e_J(X).
\]
If a partition $\mathcal P$ is contracted and loops are discarded, counting the lost vertices and edges gives
\begin{equation}\label{eq:potential-contraction}
 P(J/\mathcal P)=P(J)-\sum_{X\in\mathcal P}\sigma_J(X).
\end{equation}
Single-edge contractions preserve connectedness and bridgelessness. Boundary values on a contracted part are summed, as in Section~\ref{sec:critical-density}.

\begin{lemma}[Potential under contraction]\label{lem:critical-potential}
Let $(J,\gamma)$ be an infeasible connected bridgeless pair, and suppose $\sigma_J(X)\ge1$ for every $X\subseteq V(J)$ with $|X|\ge2$. Then $P(J)\ge6$. If $P(J)=6$, the pair is contraction-critical and $J$ is simple and $(3,6)$-tight.
\end{lemma}
\begin{proof}
Contract edges while the induced pair remains infeasible. The process ends at a contraction-critical pair $(Q,\gamma_Q)$, where $Q=J/\mathcal P$ for a partition into connected parts. By Lemma~\ref{lem:critical-reduced}, $Q$ is reduced. It has at least three vertices: the one-vertex graph and every connected bridgeless two-vertex graph realize every balanced boundary. Theorem~\ref{thm:reduced-density} therefore gives $P(Q)\ge6$. Formula~\eqref{eq:potential-contraction} gives
\[
 P(J)=P(Q)+\sum_{X\in\mathcal P}\sigma_J(X)\ge6,
\]
with strict inequality if any part has at least two vertices. Thus $P(J)=6$ forces the contraction process to make no step: $(J,\gamma)$ is already contraction-critical. Lemma~\ref{lem:critical-reduced} and Theorem~\ref{thm:reduced-density} now make $J$ simple and sparse; its potential six makes it tight.
\end{proof}

The next reduction replaces pairs of incident edges by single edges.
It produces potential five or six; the preceding lemma excludes five
for an infeasible pair and makes potential six a smaller simple tight
graph. The two cases omitted from the table, degree three with boundary
zero and degree five with boundary zero, will be handled by deletion
and Lemma~\ref{lem:tight-degree-five}, respectively.

\begin{lemma}[Splitting a low-degree vertex]\label{lem:tight-splitting}
Let $H$ be $(3,6)$-tight on $h\ge6$ vertices, let $\beta$ be a balanced
boundary, and let $u$ satisfy one of the following four cases. Choose
the indicated number of pairs of incident edges, with no edge used
twice; the pairing is arbitrary. Delete $u$ and replace each paired
path $a u b$ by a new edge $ab$, obtaining $H'$. Parallel edges are retained.
\begin{center}
\begin{tabular}{c c c c c}
\hline
$\deg_H(u)$ & $\beta(u)$ & Pairs & Unpaired edges & $P(H')$\\
\hline
$3$ & $\ne0$ & $1$ & $1$ & $5$\\
$4$ & $0$    & $2$ & $0$ & $5$\\
$4$ & $\ne0$ & $1$ & $2$ & $6$\\
$5$ & $\ne0$ & $2$ & $1$ & $6$\\
\hline
\end{tabular}
\end{center}
One can choose a balanced boundary $\beta'$ on $H'$ so that every $\beta'$-orientation lifts to a $\beta$-orientation of $H$. Moreover, $H'$ is connected and bridgeless, has the potential shown in the table, and satisfies
\[
 \sigma_{H'}(X)\ge1\quad(|X|=2),\qquad
 \sigma_{H'}(X)\ge2\quad(|X|\ge3).
\]
\end{lemma}
\begin{proof}
First orient the unpaired edges with net contribution $\beta(u)$ at $u$. One edge permits either nonzero residue, and two edges directed both inward or both outward permit the two nonzero residues. In the row with no unpaired edge, $\beta(u)=0$. Let $b(w)$ be the contribution of these edges at each vertex, and set $\beta'(w)=\beta(w)-b(w)$ for $w\ne u$. Since $b(u)=\beta(u)$ modulo three and $\sum_w b(w)=0$, the new boundary is balanced. Replacing a new directed edge $a\to b$ by $a\to u\to b$ preserves its endpoint contributions and contributes zero at $u$. Thus an orientation of $H'$ with boundary $\beta'$ lifts as claimed. If $r=\deg_H(u)$ and $\ell$ pairs are used, then $H'$ has $3h-6-r+\ell$ edges and
$P(H')=3+r-\ell$, giving the table.

To verify connectedness and bridgelessness, consider any nontrivial
cut $(X,Y)$ of $H-u$. If a shore has one vertex, its degree after deleting
$u$ is at least two. If a shore has two vertices, minimum degree three
and simplicity give at least $6-2-2=2$ edges leaving it in $H-u$.
Otherwise both shores have at least three vertices. Since
$\deg_H(u)\le5$, one shore, say $X$, contains at most two neighbours
of $u$. The corresponding cut of $H$ has at least six edges and loses
at most two when $u$ is deleted. Thus every cut of $H-u$ has at least
two edges. Adding the new edges preserves this bound.

For a two-vertex set $X$, at most one new edge has both ends in $X$, so $e_{H'}(X)\le2$ and $\sigma_{H'}(X)\ge1$. Now let $|X|\ge3$, and let $j$ count the new edges inside $X$. The new edges form a matching on the neighbours of $u$, so $X$ contains at least $2j$ such neighbours. Applying sparsity in $H$ first to $X$ and then to $X\cup\{u\}$ gives
\[
 e_{H'}(X)=e_H(X)+j\le3|X|-6+j,
\]
and
\[
 e_H(X)+2j\le3|X|-3,
 \qquad e_{H'}(X)\le3|X|-3-j.
\]
Therefore $\sigma_{H'}(X)\ge\max\{3-j,j\}\ge2$.
\end{proof}

\subsection{Completing the induction}

\begin{proof}[Proof of Theorem~\ref{thm:tight-strong}]
We induct on $h=|V(H)|\ge5$.

\emph{Base case.} For $h=5$, the graph is $K_5$ minus an edge and
contains a spanning wheel with four rim vertices. This wheel is strongly
$\Z_3$-connected~\cite[Proposition~3.6]{Lai00}; we include a direct proof.

Fix any balanced boundary $\beta$ on the wheel. Label its rim $v_1,v_2,v_3,v_4$ cyclically, give the rim a clockwise reference orientation, and direct the reference spokes toward the hub. Let $x_i\in\{1,-1\}$ be the value on the rim edge from $v_i$ to $v_{i+1}$. Its spoke must then have value
\[
 s_i=\beta(v_i)-x_i+x_{i-1}\quad\text{in }\Z_3,
\]
with indices modulo four. If $\beta(v_i)=0$, the condition $s_i\ne0$ requires a change of sign between $x_{i-1}$ and $x_i$. If $\beta(v_i)\ne0$, keeping the sign always works, and exactly one of the two possible changes works. When the number of zero rim boundaries is even, change sign exactly at those vertices. When it is odd, also change at one vertex with nonzero boundary, and choose the initial sign to make that additional change allowed. The number of changes is even, so the signs close around the rim. Each spoke value is nonzero, and the hub equation follows from balance. Interpreting $1$ and $-1$ as keeping and reversing the reference directions gives the required orientation. Finally, orient the extra edge of $K_5$ minus an edge arbitrarily and apply the wheel argument to the residual boundary.

\emph{Induction step.} Let $h\ge6$ and suppose that a balanced boundary
$\beta$ is not realizable. The minimum degree is at least three, and
$2|E(H)|=6h-12$ ensures a vertex $u$ of degree three, four or five.
If $\deg_H(u)=5$ and $\beta(u)=0$, Lemma~\ref{lem:tight-degree-five}
is a contradiction. If $\deg_H(u)=3$ and $\beta(u)=0$, orient all three
incident edges away from $u$; the remaining graph is tight on $h-1$
vertices, and induction realizes the balanced residual boundary.

Every other case is covered by Lemma~\ref{lem:tight-splitting}. It produces an infeasible pair $(H',\beta')$ with $P(H')\in\{5,6\}$ satisfying the hypotheses of Lemma~\ref{lem:critical-potential}. Potential five is impossible. Potential six makes $H'$ simple and tight on $h-1\ge5$ vertices, where induction again gives the required orientation. This proves strong $\Z_3$-connectivity.

\emph{Density conclusion.} By Theorem~\ref{thm:reduced-density}, a
reduced graph on $n\ge5$ vertices is simple and sparse. Equality
$m=3n-6$ would make it tight, hence strongly $\Z_3$-connected,
contradicting reducedness. Thus $m\le3n-7$.
Lemma~\ref{lem:critical-reduced} applies this conclusion to
three-flow-critical graphs.
\end{proof}

Together with Proposition~\ref{prop:critical-small-orders}, this shows
that $K_4$ is the only three-flow-critical graph attaining $3n-6$.
The six-vertex example in that proposition attains the improved bound
$3n-7$.

\clearpage
\section*{AI methodology}
The author used ChatGPT, including Codex, to assist with literature
searches, exposition, proof organization, and the development and checking
of proof arguments throughout the paper, particularly those in
Section~\ref{sec:critical-density} and Appendix~\ref{app:critical-density}.
The author verified the correctness of the content, the originality
of the contributions, and the accuracy of all references.

\section*{Acknowledgments}
The author thanks Noga Alon for insightful correspondence and for
drawing his attention to the paper by Alon, Buci\'c, and Davies~\cite{ABD}.
He is also grateful to M.~J.~Moghaddamzadeh for sharing some of his
ideas on flows in graphs.

\begingroup
\raggedright
\setlength{\parskip}{0pt}

\endgroup


\begin{thebibliography}{CLSW+26}
\interlinepenalty=10000
\setlength{\itemsep}{0pt}
\setlength{\parsep}{0pt}
\setlength{\parskip}{0pt}
\bibitem[ABD26]{ABD}
N. Alon, M. Buci\'c and J. Davies,
\emph{Hadwiger--Nelson Problem for Typical Norms},
\href{https://arxiv.org/abs/2609.36220v1}{arXiv:2609.36220v1} (2026), Lemma~7.

\bibitem[ADL+26]{ADLMS}
A. S. \'{A}rnad\'ottir, Z. Dvo\v{r}\'ak, B. Lidick\'y, B. Moore,
E. Smith-Roberge and R. \v{S}\'amal,
\emph{Flow-critical graphs}, Journal of Graph Theory (2026), Early View, 1--41.
\href{https://doi.org/10.1002/jgt.70143}{doi:10.1002/jgt.70143}.

\bibitem[BG14]{BG}
J. Bar\'at and D. Gerbner,
\emph{Edge-decomposition of graphs into copies of a tree with four edges},
Electronic Journal of Combinatorics \textbf{21}(1) (2014), Paper P1.55.
\href{https://www.combinatorics.org/ojs/index.php/eljc/article/view/v21i1p55}{Published article}, Conjecture~6.

\bibitem[Bar98]{Barthe}
F. Barthe, \emph{On a reverse form of the Brascamp--Lieb inequality},
Inventiones Mathematicae \textbf{134} (1998), 335--361.
\href{https://doi.org/10.1007/s002220050267}{doi:10.1007/s002220050267}.

\bibitem[CN16]{ChenNing}
F. Chen and B. Ning, \emph{A note on nowhere-zero $3$-flows and $\mathbb Z_3$-connectivity},
Ars Mathematica Contemporanea \textbf{10} (2016), 91--98.
\href{https://doi.org/10.26493/1855-3974.676.1a9}{doi:10.26493/1855-3974.676.1a9}.
\href{https://arxiv.org/abs/1406.1554}{arXiv:1406.1554}, Theorem~1.17 of the arXiv version.

\bibitem[CLS+26]{CLSWWei}
D. W. Cranston, J. Li, B. Su, Z. Wang and C. Wei,
\emph{Orientations of $10$-edge-connected planar multigraphs and applications},
Journal of Graph Theory (2026), Early View, 1--25.
\href{https://doi.org/10.1002/jgt.70115}{doi:10.1002/jgt.70115}.

\bibitem[Cun85]{Cunningham}
W. H. Cunningham, \emph{Optimal attack and reinforcement of a network},
Journal of the ACM \textbf{32} (1985), 549--561.
\href{https://doi.org/10.1145/3828.3829}{doi:10.1145/3828.3829}.

\bibitem[DRS17]{DFRS}
M. DeVos, E. Rollov\'a and R. \v{S}\'amal,
\emph{A new proof of Seymour's $6$-flow theorem},
Journal of Combinatorial Theory, Series B \textbf{122} (2017), 187--195.
\href{https://doi.org/10.1016/j.jctb.2016.05.010}{doi:10.1016/j.jctb.2016.05.010}.
\href{https://arxiv.org/abs/1512.06214}{arXiv:1512.06214}.

\bibitem[DM23]{DMdensity}
Z. Dvo\v{r}\'ak and B. Mohar,
\emph{On density of $\mathbb Z_3$-flow-critical graphs},
SIAM Journal on Discrete Mathematics \textbf{37} (2023), 699--717.
\href{https://doi.org/10.1137/22M1496529}{doi:10.1137/22M1496529}.
\href{https://arxiv.org/abs/2205.07498v3}{arXiv:2205.07498v3}, Conjecture~1.3.

\bibitem[Edm65]{Edmonds}
J. Edmonds, \emph{Minimum partition of a matroid into independent subsets},
Journal of Research of the National Bureau of Standards, Section B \textbf{69B} (1965), 67--72.
\href{https://doi.org/10.6028/jres.069B.004}{doi:10.6028/jres.069B.004}.

\bibitem[EJLT18]{EJLT}
L. Esperet, R. de Joannis de Verclos, T.-N. Le and S. Thomass\'e,
\emph{Additive bases and flows in graphs}, SIAM J. Discrete Math. \textbf{32} (2018), 534--542.
\href{https://doi.org/10.1137/17M1113758}{doi:10.1137/17M1113758}.

\bibitem[FJS01]{FJS}
A. Frank, T. Jord\'an and Z. Szigeti, \emph{An orientation theorem with parity conditions}, Discrete Appl. Math. \textbf{115} (2001), 37--47.
\href{https://doi.org/10.1016/S0166-218X(01)00213-X}{doi:10.1016/S0166-218X(01)00213-X}.

\bibitem[FK02]{FK}
A. Frank and Z. Kir\'aly, \emph{Graph orientations with edge-connection and parity constraints}, Combinatorica \textbf{22} (2002), 47--70.
\href{https://doi.org/10.1007/s004930200003}{doi:10.1007/s004930200003}.

\bibitem[Gab23]{Gabow}
H. N. Gabow, \emph{A weight-scaling algorithm for $f$-factors of multigraphs},
Algorithmica \textbf{85} (2023), 3214--3289.
\href{https://doi.org/10.1007/s00453-023-01127-x}{doi:10.1007/s00453-023-01127-x}.
The factor polytope is stated in Appendix~C of
\href{https://arxiv.org/abs/2010.01102}{arXiv:2010.01102}, p.~56.

\bibitem[Hak65]{Hakimi}
S. L. Hakimi, \emph{On the degrees of the vertices of a directed graph},
Journal of the Franklin Institute \textbf{279} (1965), 290--308.
\href{https://doi.org/10.1016/0016-0032(65)90340-6}{doi:10.1016/0016-0032(65)90340-6}.

\bibitem[HLL18]{HLL}
M. Han, H.-J. Lai and J. Li, \emph{Nowhere-zero $3$-flow and $\mathbb Z_3$-connectedness in graphs with four edge-disjoint spanning trees}, J. Graph Theory \textbf{88} (2018), 577--591.
\href{https://doi.org/10.1002/jgt.22231}{doi:10.1002/jgt.22231}.

\bibitem[HLWZ18]{HLWZ}
M. Han, J. Li, Y. Wu and C.-Q. Zhang,
\emph{Counterexamples to Jaeger's Circular Flow Conjecture},
J. Combin. Theory Ser. B \textbf{131} (2018), 1--11.
\href{https://doi.org/10.1016/j.jctb.2018.01.002}{doi:10.1016/j.jctb.2018.01.002}.

\bibitem[Has16]{Hasanvand2016}
M. Hasanvand, \emph{Tutte's $3$-Flow Conjecture in $3$-tree-connected graphs},
\href{https://arxiv.org/abs/1611.02231v2}{arXiv:1611.02231v2} (2016; revised 2022), Conjecture~3.

\bibitem[Has24]{Hasanvand2024}
M. Hasanvand, \emph{Modulo orientations with bounded out-degrees},
Discrete Mathematics \textbf{347} (2024), Article 113634.
\href{https://doi.org/10.1016/j.disc.2023.113634}{doi:10.1016/j.disc.2023.113634}.

\bibitem[Jae88]{Jaeger88}
F. Jaeger, \emph{Nowhere-zero flow problems}, in:
Selected Topics in Graph Theory 3, Academic Press, 1988, pp. 71--95.

\bibitem[JLPT92]{JLPT}
F. Jaeger, N. Linial, C. Payan and M. Tarsi,
\emph{Group connectivity of graphs---a nonhomogeneous analogue of nowhere-zero flow properties},
Journal of Combinatorial Theory, Series B \textbf{56} (1992), 165--182.
\href{https://doi.org/10.1016/0095-8956(92)90016-Q}{doi:10.1016/0095-8956(92)90016-Q}.

\bibitem[Jaf26]{Jafari}
A. Jafari, \emph{Graphic additive bases and graph flows},
\href{https://arxiv.org/abs/2610.07034v3}{arXiv:2610.07034v3} (2026).

\bibitem[Kir14]{Kiraly}
Cs. Kir\'aly, \emph{Algorithms for finding a rooted $(k,1)$-edge-connected orientation},
Discrete Applied Mathematics \textbf{166} (2014), 263--268.
\href{https://doi.org/10.1016/j.dam.2013.10.004}{doi:10.1016/j.dam.2013.10.004}.

\bibitem[Lai00]{Lai00}
H.-J. Lai, \emph{Group connectivity of $3$-edge-connected chordal graphs},
Graphs and Combinatorics \textbf{16} (2000), 165--176.
\href{https://doi.org/10.1007/s003730050014}{doi:10.1007/s003730050014}.

\bibitem[LMS+22]{LMSWW}
J. Li, Y. Ma, Y. Shi, W. Wang and Y. Wu,
\emph{On $3$-flow-critical graphs}, European Journal of Combinatorics
\textbf{100} (2022), Article 103451.
\href{https://doi.org/10.1016/j.ejc.2021.103451}{doi:10.1016/j.ejc.2021.103451}.

\bibitem[LTWZ18]{LTWZindex}
J. Li, C. Thomassen, Y. Wu and C.-Q. Zhang,
\emph{The flow index and strongly connected orientations},
European Journal of Combinatorics \textbf{70} (2018), 164--177.
\href{https://doi.org/10.1016/j.ejc.2017.12.009}{doi:10.1016/j.ejc.2017.12.009}.

\bibitem[LTWZ13]{LTWZ}
L. M. Lov\'asz, C. Thomassen, Y. Wu and C.-Q. Zhang,
\emph{Nowhere-zero $3$-flows and modulo $k$-orientations},
Journal of Combinatorial Theory, Series B \textbf{103} (2013), 587--598.
\href{https://doi.org/10.1016/j.jctb.2013.06.003}{doi:10.1016/j.jctb.2013.06.003}.

\bibitem[LMTZ26]{LMTZ}
Y. Lu, Z. Miao, Z. Tian and C.-Q. Zhang,
\emph{Modulo $k$-orientations of highly edge-connected signed graphs},
\href{https://arxiv.org/abs/2610.42019v1}{arXiv:2610.42019v1} (2026), Theorem~1.4.

\bibitem[NW60]{NW}
C. St. J. A. Nash-Williams,
\emph{On orientations, connectivity and odd-vertex-pairings in finite graphs},
Canadian Journal of Mathematics \textbf{12} (1960), 555--567.
\href{https://doi.org/10.4153/CJM-1960-049-6}{doi:10.4153/CJM-1960-049-6}.

\bibitem[NW61]{NWtrees}
C. St. J. A. Nash-Williams, \emph{Edge-disjoint spanning trees of finite graphs},
Journal of the London Mathematical Society \textbf{36} (1961), 445--450.
\href{https://doi.org/10.1112/jlms/s1-36.1.445}{doi:10.1112/jlms/s1-36.1.445}.

\bibitem[Sey81]{Seymour}
P. D. Seymour, \emph{Nowhere-zero $6$-flows},
Journal of Combinatorial Theory, Series B \textbf{30} (1981), 130--135.
\href{https://doi.org/10.1016/0095-8956(81)90058-7}{doi:10.1016/0095-8956(81)90058-7}.

\bibitem[Tut61]{Tutte}
W. T. Tutte, \emph{On the problem of decomposing a graph into $n$ connected factors},
Journal of the London Mathematical Society \textbf{36} (1961), 221--230.
\href{https://doi.org/10.1112/jlms/s1-36.1.221}{doi:10.1112/jlms/s1-36.1.221}.

\bibitem[Zhu15]{Zhu}
X. Zhu, \emph{Circular flow number of highly edge connected signed graphs},
Journal of Combinatorial Theory, Series B \textbf{112} (2015), 93--103.
\href{https://doi.org/10.1016/j.jctb.2014.12.002}{doi:10.1016/j.jctb.2014.12.002}.

\end{thebibliography}
\end{document}